\documentclass[10pt,reqno]{amsart}
\usepackage{amsmath,amsthm,amssymb,mathtools,amscd}
\usepackage{latexsym,graphicx,color,indentfirst}
\usepackage{algorithmic}
\usepackage[linesnumbered,boxed,norelsize]{algorithm2e}
\usepackage{fancyhdr}
\usepackage{mathrsfs}
\usepackage{float}
\usepackage{float}
\usepackage{color}
\usepackage{listings}
\usepackage{multirow}
\usepackage{graphicx}
\usepackage{mathrsfs}
\usepackage{ifpdf}
\usepackage{epstopdf}
\usepackage{enumitem}

\newtheorem{theorem}{Theorem}[section]
\newtheorem{lemma}[theorem]{Lemma}
\newtheorem{assumption}[theorem]{Assumption}

\newtheorem{proposition}[theorem]{Proposition}

\newtheorem{remark}[theorem]{Remark}

\newcommand{\eps}{\varepsilon}

\date{\today}

\def\eps{\varepsilon}
\def\geq{\geqslant}
\def\ge{\geqslant}

\def\l{\left}
\def\le{\leqslant}
\def\leq{\leqslant}
\def\r{\right}

\def\p{\partial}
\def\d{\delta}

\def\l{\langle}
\def\r{\rangle}
\def\N{\mathcal N}
\def\R{\mathcal R}
\def\D{\mathscr D}
\def\X{\mathcal X}
\def\Y{\mathcal Y}

\def\a{\alpha}

\def\d{\delta}
\def\la{\lambda}
\def\X{\mathcal X}
\def\Y{\mathcal Y}
\def\p{\partial}
\def\d{\delta}

\def\l{\langle}
\def\r{\rangle}
\def\N{\mathcal N}
\def\R{\mathcal R}
\def\D{\mathscr D}
\def\X{\mathcal X}
\def\Y{\mathcal Y}

\def\a{\alpha}

\def\d{\delta}
\def\la{\lambda}
\def\X{\mathcal X}
\def\Y{\mathcal Y}

\title[Two-point gradient method]{Regularization of inverse problems by two-point gradient methods with convex constraints}
\author{Min Zhong}
\address{Department of Mathematics, Southeast University, Nanjing, Jiangsu 210096, People's Republic of China}
\email{\tt min.zhong@seu.edu.cn}
\author{Wei Wang }
\address{Corresponding author. College of Mathematics, Physics and Information Engineering, Jiaxing University, Zhejiang 314001, People's Republic of China}
\email{\tt weiwang${}_{-}$math@126.com}
\author{Qinian Jin}
\address{Mathematical Sciences Institute, Australian National University, Canberra, ACT 2601, Australia}
\email{\tt Qinian.Jin@anu.edu.au}


\date{}

\begin{document}

\begin{abstract}
In this paper, we propose and analyze a two-point gradient  method for solving inverse problems in Banach spaces which
is based on the Landweber iteration and an extrapolation strategy.  The method allows to use non-smooth penalty terms,
including the $L^1$ and the total variation-like penalty functionals, which are significant in reconstructing special
features of solutions such as sparsity and piecewise constancy in practical applications. The design of the method involves
the choices of the step sizes and the combination parameters which are carefully discussed. Numerical simulations are presented
to illustrate the effectiveness of the proposed method.
\end{abstract}

\maketitle

\section{\bf Introduction}
\setcounter{equation}{0}

In this paper we are interested in solving inverse problems of the form
\begin{align}\label{equation}
F(x)=y\,,
\end{align}
where $F:\D(F)\subset \X\rightarrow \Y$ is a  Fr\'{e}chet differentiable operator between two
Banach spaces $\X$ and $\Y$. Throughout this paper we will assume \eqref{equation} has a solution,
which is not necessarily unique. Such inverse problems are ill-posed in the sense of unstable dependence
of solutions on small perturbations of the data. Instead of exact data $y$, in practice we are
given only noisy data $y^\d$ satisfying
\begin{align}\label{noise}
\|y-y^\d\|\leq\d\,.
\end{align}
Consequently, it is necessary to apply regularization methods to solve (\ref{equation}) approximately (\cite{ehn96}).

Landweber iteration is one of the most prominent regularization methods for solving inverse problems
formulated in Hilbert spaces. A complete analysis on this method for linear problems as well as nonlinear
problems can be found in \cite{ehn96,hns95}. This method has received tremendous attention due to
its simple implementation and robustness with respect to noise.

The classical Landweber iteration in Hilbert spaces, however, has the tendency to over-smooth solutions, which makes it
difficult to capture special features of the sought solution such as sparsity and discontinuity. To overcome this drawback,
various reformulations of Landweber iteration either in Banach spaces or in a manner of incorporating general non-smooth
convex penalty functionals have been proposed, see \cite{sls06,kss09,hk10,bh12,j12,jw13,wwh18}. In \cite{sls06,kss09}, using
a gradient method for solving
\begin{align}\label{Phir}
\min \frac{1}{s}\|F(x)-y^\d\|^s,
\end{align}
the authors proposed the Landweber iteration of the form
\begin{equation}\label{lenear1}
\begin{split}
\xi_{n+1}^\d &= \xi_n^\d-\mu_n^\d F'(x_n^\d)^*J_s^{\Y}(F(x_n^\d)-y^\d)\,,\nonumber\\
x_{n+1}^\d &= J_q^{\X^*}(\xi_{n+1}^\d)\,,
\end{split}
\end{equation}
for solving linear as well as nonlinear inverse problems in Banach spaces, assuming suitable smoothness and convexity on
$\X$ and $\Y$, where $F'(x)$ and $F'(x)^*$ denote the Fr\'{e}chet derivative of $F$ at $x$ and its adjoint, $\mu_n^\d$ denotes the step size,
and $J_s^{\Y}:\Y\rightarrow \Y^*$ and $J_q^{\X^*}:\X^*\rightarrow \X$ with $1<s,q<\infty$ denote the duality mappings with gauge functions
$t\rightarrow t^{s-1}$ and $t\rightarrow t^{q-1}$ respectively. This formulation of Landweber iteration, however, exclude the use of
the $L^1$ and the total variation like penalty functionals. A Landweber-type iteration with general uniform convex penalty functionals was
introduced in \cite{bh12} for solving linear inverse problems and was extended in \cite{jw13} for solving nonlinear inverse problems. Let $\Theta:\X\rightarrow(-\infty,\infty]$ be a proper lower semi-continuous uniformly convex functional, the method in \cite{bh12,jw13}
can be formulated as
\begin{equation}\label{linear2}
\begin{split}
\xi_{n+1}^\d &= \xi_n^\d-\mu_n^\d F'(x_n^\d)^*J_s^{\Y}(F(x_n^\d)-y^\d),\\
x_{n+1}^\d &= \textrm{arg}\min_{x\in \X} \left\{\Theta(x)-\langle \xi_{n+1}^\d,x\rangle\right\}.
\end{split}
\end{equation}
The advantage of this method is the freedom on the choice of $\Theta$ so that it can be utilized in detecting different
features of the sought solution.

It is well known that Landweber iteration is a slowly convergent method. As alternatives to Landweber iteration, one may consider the
second order iterative methods, such as the Levenberg-Maquardt method \cite{H1997,jy16}, the iteratively regularized Gauss-Newton
method \cite{jz13,k15}, or the nonstationary iterated Tikhonov regularization \cite{jz14}. The advantage of these methods is that
they require less number of iterations to satisfy the respective stopping rule than the Landweber iteration, however they always require
to spend more computational time in dealing with each iteration step. Therefore, it becomes more desirable to accelerate Landweber
iteration by preserving its simple implementation feature.

For linear inverse problems in Hilbert spaces, a family of accelerated Landweber iterations were proposed in \cite{H1991} using
the orthogonal polynomials and the spectral theory of self-adjoint operators. The acceleration strategy using orthogonal polynomials
is no longer available for Landweber iteration in Banach spaces with general convex penalty functionals. In \cite{hk10} an acceleration of
Landweber iteration in Banach spaces was considered based on choosing optimal step size in each iteration step. In \cite{HJW2015,sls09}
the sequential subspace optimization strategy was employed to accelerate the Landweber iteration.

The Nesterov's strategy was proposed in \cite{n83} to accelerate gradient method. It has played an important role on the development of
fast first order methods for solving well-posed convex optimization problems \cite{AP2016,bt09}. Recently, an accelerated version of Landweber iteration
based on Nesterov's strategy was proposed  in \cite{j16} which includes the following method
\begin{equation}\label{nesterov}
\begin{split}
z_n^\d &= x_n^\d+\frac{n}{n+\alpha}(x_n^\d-x_{n-1}^\d),\\
x_{n+1}^\d &= z_n^\d +\mu_n^\d F'(z_n^\d)^*(y^\d-F(z_n^\d))
\end{split}
\end{equation}
with $\a\ge 3$ as a special case for solving ill-posed inverse problems (\ref{equation}) in Hilbert spaces, where $x_{-1}^\d = x_0^\d=x_0$
is an initial guess. Although no convergence analysis for \eqref{nesterov} could be given, the numerical results presented in \cite{j16,zw19}
clearly demonstrate its usefulness and acceleration effect. By replacing $n/(n+\a)$ in (\ref{nesterov}) by a general connection
parameters $\lambda_n^\d$, a so called two-point gradient method was proposed in \cite{hr17} and a convergence result was proved under
a suitable choice of $\{\lambda_n^\d\}$. Furthermore, based on the assumption of local convexity of the residual functional around the sought
solution, a weak convergence result on (\ref{nesterov}) was proved in \cite{hr18} recently.

In this paper, by incorporating an extrapolation step into the iteration scheme (\ref{linear2}), we propose a two-point gradient method
for solving inverse problems in Banach spaces with a general uniformly convex penalty term $\Theta$, which takes the form
\begin{align*}
\zeta_n^\d &= \xi_n^\d + \lambda_n^\d (\xi_n^\d - \xi_{n-1}^\d),\\
z_n^\d & = \arg\min_{z\in \X} \left\{\Theta(z) - \l \zeta_n^\d, z\r\right\},\\
\xi_{n+1}^\d & = \zeta_n^\d - \mu_n^\d F'(z_n^\d)^* J_s^{\Y} (F(z_n^\d)-y^\d)
\end{align*}
with suitable step sizes $\mu_n^\d$ and combination parameters $\lambda_n^\d$; after terminated by a discrepancy principle,
we then use
$$
x_n^\d := \arg\min_{x\in \X} \left\{ \Theta(x) - \l \xi_n^\d, x\r \right\}
$$
as an approximate solution.  We note that when $\lambda_n^\d=0$, our method becomes the Landweber iteration of the form (\ref{linear2})
and when $\lambda_n^\d = n/(n+\a)$ it becomes a refined version of the Nesterov acceleration of Landweber iteration proposed in \cite{j16}.
We note also that, when both $\X$ and $\Y$ are Hilbert spaces and $\Theta(x)= \|x\|^2/2$, our method becomes the two-point gradient
methods introduced in \cite{hr17} for solving inverse problems in Hilbert spaces. Unlike the method in \cite{hr17}, our method not only
works for inverse problems in Banach spaces, but also allows the use of general convex penalty functions including the $L^1$ and the total
variation like functions.  Due to the possible non-smoothness of $\Theta$ and the non-Hilbertian structures of $\X$ and $\Y$, we need to
use tools from convex analysis and subdifferential calculus to carry out the convergence analysis. Under certain conditions on the
combination parameters $\{\lambda_n^\d\}$, we obtain a convergence result on our method. In order to find nontrivial $\lambda_n^\d$,
we adapt the discrete backtracking search (DBTS) algorithm in \cite{hr17} to our situation and provide a complete convergence analysis of
the corresponding method by showing a uniform convergence result for the noise-free counterpart with the combination parameters chosen in
a certain range. Our analysis in fact improves the convergence result in \cite{hr17} by removing the technical conditions on $\{\lambda_n^\d\}$
chosen by the DBTS algorithm.

The paper is organized as follows, In section 2, we give some preliminaries from convex analysis. In section 3, we
formulate our two-point gradient method with a general uniformly convex penalty term and present
the detailed convergence analysis. We also discuss the choices of the combination parameters by a discrete
backtracking search algorithm. Finally in section 4, numerical simulations are given to test the performance of the method.

\section{\bf Preliminaries}\label{sec:prelim}
\setcounter{equation}{0}

In this section, we introduce some necessary concepts and properties related to Banach space and convex analysis,
we refer to \cite{ZA02} for more details.

Let $\X$ be a Banach space whose norm is denoted by $\|\cdot\|$, we use $\X^*$ to denote its dual space. Given $x\in \X$
and $\xi \in \X^*$, we write $\l \xi, x\r=\xi (x)$ for the duality pairing. For a bounded linear operator $A: \X \to \Y$ between
two Banach spaces $\X$ and $\Y$, we use $A^*: \Y^*\to \X^*$ to denote its adjoint.

Given a convex function $\Theta: \X \to (-\infty, \infty]$, we use $ \p \Theta(x)$ to denote the subdifferential
of $\Theta$  at $x\in \X$, i.e.
$$
\p \Theta(x) := \{\xi \in \X^*: \Theta(\bar x) -\Theta(x) -\l \xi, \bar x-x\r\ge 0 \mbox{ for all } \bar x\in \X\}.
$$
Let $\D(\Theta): = \{x\in \X:\Theta(x)<\infty\}$ be its effective domain and let
\begin{align*}
\D(\p\Theta): = \{x\in \D(\Theta):\p\Theta(x)\neq\varnothing\}.
\end{align*}
 The Bregman distance induced by $\Theta$ at $x$
in the direction $\xi\in \p\Theta(x) $  is defined by
\begin{equation*}
D_\xi \Theta(\bar x,x):=\Theta(\bar x)-\Theta(x)-\l \xi, \bar x-x\r, \qquad \forall \bar x\in  \X
\end{equation*}
which is always nonnegative and satisfies the identity
\begin{equation}\label{2.1}
D_{\xi_2} \Theta(x,x_2)-D_{\xi_1} \Theta(x, x_1) =D_{\xi_2} \Theta(x_1,x_2) +\l \xi_2-\xi_1, x_1-x\r
\end{equation}
for all $x\in \D(\Theta), x_1, x_2\in \D(\p \Theta)$, and $\xi_1\in \p \Theta(x_1)$, $\xi_2\in \p \Theta(x_2)$.

A proper convex function $\Theta: \X \to (-\infty, \infty]$ is called uniformly convex if there exists
a strictly increasing function $h:[0,\infty)\rightarrow [0,\infty)$ with $h(0) = 0$ such that
\begin{equation} \label{2.2}
\Theta(\la \bar x+(1-\la)x) +\la (1-\la)h(\|x-\bar x\|) \le \la \Theta(\bar x) +(1-\la) \Theta(x)
\end{equation}
for all $\bar x,x\in\X$ and $\la\in(0,1)$. If $h(t) = c_0t^p$ for some $c_0>0$ and $p\geq 2$ in (\ref{2.2}),
then $\Theta$ is called $p$-convex. It can be shown that $\Theta$ is $p$-convex if and only if
\begin{align}\label{pconvex}
D_\xi \Theta(\bar x,x)\geq c_0\|x-\bar x\|^p,\quad\forall \bar x\in \X, \ x\in \D(\p\Theta), \ \xi\in\p\Theta(x)\,.
\end{align}

For a proper lower semi-continuous convex function $\Theta: \X \to (-\infty, \infty]$, its Legendre-Fenchel conjugate
is defined by
\begin{equation*}
\Theta^*(\xi):=\sup_{x\in \X} \left\{\l\xi, x\r -\Theta(x)\right\}, \quad \xi\in \X^*
\end{equation*}
which is also proper, lower semi-continuous, and convex. If $\X$ is reflexive, then
\begin{equation}\label{2.3}
\xi\in \p \Theta(x) \Longleftrightarrow x\in \p \Theta^*(\xi) \Longleftrightarrow \Theta(x) +\Theta^*(\xi) =\l \xi, x\r.
\end{equation}
Moreover, if $\Theta$ is $p$-convex with $p\geq 2$ then it follows from \cite[Corollary 3.5.11]{ZA02} that $\D(\Theta^*)=\X^*$,
$\Theta^*$ is Fr\'{e}chet differentiable and its gradient $\nabla \Theta^*:  X^*\to \X$ satisfies
\begin{equation}\label{2.4}
\|\nabla \Theta^*(\xi_1)-\nabla \Theta^*(\xi_2) \|\le \left(\frac{\|\xi_1-\xi_2\|}{2c_0}\right)^{\frac{1}{p-1}},
\quad \forall \xi_1, \xi_2\in \X^*.
\end{equation}
Consequently, it follows from (\ref{2.3}) that
\begin{equation}\label{2.5}
x=\nabla \Theta^*(\xi) \Longleftrightarrow \xi \in \p \Theta(x) \Longleftrightarrow x =\arg \min_{z\in \X} \left\{ \Theta(z) -\l \xi, z\r\right\}.
\end{equation}

\begin{lemma}\label{lem:2.1}
If $\Theta$ is $p$-convex with $p\geq 2$, then for any pairs $(x,\xi)$ and $(\bar x,\bar \xi) $
with $x, \bar x \in \D(\p\Theta), \xi\in\p\Theta(x), \bar \xi \in\p\Theta(\bar{x})$, we have
\begin{align}\label{pconv2}
D_\xi \Theta(\bar x,x) \le \frac{1}{p^*(2c_0)^{p*-1}}\|\xi-\bar \xi\|^{p*}\,,
\end{align}
where $p^*: = p/(p-1)$.
\end{lemma}

\begin{proof}
Applying \eqref{2.3},  $\bar x = \nabla \Theta^*(\bar \xi)$ and \eqref{2.4}, it follows that
\begin{align*}
D_\xi \Theta(\bar x,x)
& = \Theta^*(\xi)-\Theta^*(\bar \xi)-\l  \xi-\bar \xi, \nabla \Theta^*(\bar \xi)\r \displaybreak[0]\\
& = \int_0^1 \l \xi -\bar \xi, \nabla \Theta^* (\bar \xi +t (\xi-\bar \xi)) -\nabla \Theta^*(\bar \xi)\r d t \displaybreak[0]\\
& \le \|\xi-\bar \xi\| \int_0^1  \| \nabla \Theta^* (\bar \xi +t (\xi-\bar \xi)) -\nabla \Theta^*(\bar \xi)\| dt\\
& \le  \frac{1}{p^*(2c_0)^{p*-1}}\|\xi-\bar \xi\|^{p*}
\end{align*}
which shows the result.
\end{proof}

On a Banach space $\X$, we consider for $1<s<\infty$ the convex function $x\to \|x\|^s/s$.
Its subgradient at $x$ is given by
\begin{equation}\label{Jr}
J_s^{\X}(x):=\left\{\xi\in \X^*: \|\xi\|=\|x\|^{s-1} \mbox{ and } \l \xi, x\r=\|x\|^s\right\}
\end{equation}
which gives the duality mapping $J_s^{\X}: \X \to 2^{\X^*}$ of $\X$ with gauge function $t\to t^{s-1}$.
If $\X$ is uniformly smooth in the sense that its modulus of smoothness
\begin{equation*}
\rho_{\X}(t) := \sup\{\|\bar x+ x\|+\|\bar x-x\|- 2 : \|\bar x\| = 1,  \|x\|\le t\}
\end{equation*}
satisfies $\lim_{t\searrow 0} \frac{\rho_{\X}(t)}{t} =0$, then the duality mapping $J_s^{\X}$, for each $1<s<\infty$,
is single valued and uniformly continuous on bounded sets. There are many examples of uniformly smooth Banach spaces,
e.g., sequence space $\ell^s$, Lebesgue space $L^s$, Sobolev space $W^{k,s}$ and Besov space $B^{q,s}$ with $1<s<\infty$.

\section{\bf The two-point gradient method}
\setcounter{equation}{0}

We consider
\begin{equation}\label{sys}
F(x) =y,
\end{equation}
where $F: \D(F)\subset \X\to \Y$ is an operator between two Banach spaces $\X$ and $\Y$. Throughout this paper,
we will always assume that $\X$ is reflexive, $\Y$ is uniformly smooth, and \eqref{sys} has a solutions.
In order to capture the special feature of the sought solution, we will use a general convex function $\Theta: \X \to (-\infty, \infty]$
as a penalty term. We will need a few assumptions concerning $\Theta$ and $F$.

\begin{assumption}\label{A1}
$\Theta:\X\to (-\infty, \infty]$ is proper, weakly lower semi-continuous, and $p$-convex with $p\geq 2$ such that the condition
\eqref{pconvex} is satisfied for some $c_0>0$.
\end{assumption}

\begin{assumption}\label{A2}
\begin{enumerate}[leftmargin = 0.8cm]

\item[(a)] There exists $\rho > 0$, $x_0\in \X$ and $\xi_0\in \p\Theta(x_0)$ such that $B_{3\rho}(x_0)\subset \D(F)$
and (\ref{sys}) has a solution $x^*\in \D(\Theta)$ with
$$
D_{\xi_0} \Theta(x^*, x_0) \le c_0 \rho^p,
$$
where $B_{\rho}(x_0)$ denotes the closed ball around $x_0$ with radius $\rho$.

\item[(b)] The operator $F$ is weakly closed on $\D(F)$.

\item[(c)]  There exists a family of bounded linear operators $\{L(x): \X \to \Y\}_{x\in B_{3\rho}(x_0)}$ such that
$x\to L(x)$ is continuous on $B_{3\rho}(x_0)$ and there is $0\le \eta <1$ such that
\begin{equation*}
\|F(x)-F(\bar x) -L(\bar x) (x-\bar x)\|\le \eta \|F(x) -F(\bar x)\|
\end{equation*}
for all $x, \bar x \in B_{3\rho}(x_0)$. Moreover, there is a constant $C_0>0$ such that
$$
\|L(x) \|_{\X\to \Y} \le C_0, \quad \forall x \in B_{3\rho}(x_0).
$$
\end{enumerate}
\end{assumption}

All the conditions in Assumption 3.2 are standard. The condition (c) is called the tangential condition and is widely used in the
analysis of iterative regularization methods for nonlinear ill-posed inverse problems \cite{hns95}. The weak closedness of $F$ in
condition (b) means that for any sequence $\{x_n\}$ in $\D(F)$ satisfying $x_n\rightharpoonup x$ and $F(x_n)\rightharpoonup v$,
then $x\in \D(F)$ and $F(x) = v$, where we use ``$\rightharpoonup$" to denote the weak convergence.

Using the $p$-convex function $\Theta$ specified in Assumption \ref{A1}, we may pick among solutions of (\ref{sys}) the one
with the desired feature. We define $x^{\dag}$ to be a solution of (\ref{sys}) with the property
\begin{equation}\label{xdag}
 D_{\xi_0} \Theta(x^{\dag},x_0) := \min_{x\in \D(\Theta)\cap \D(F) }\left\{D_{\xi_0} \Theta(x,x_0) : F(x) = y \right\}\,.
\end{equation}
When $\X$ is reflexive, by using the weak closedness of $F$ and the weak lower semi-continuity of $\Theta$, it is standard
to show that $x^\dag$ exists. According to Assumption \ref{A2} (a), we always have
$$
D_{\xi_0} \Theta(x^\dag, x_0) \le c_0 \rho^p
$$
which together with Assumption \ref{A1} implies that $\|x^\dag-x_0\| \le \rho$. The following lemma shows that $x^\dag$ is
uniquely defined.

\begin{lemma} \label{lem0}
Under Assumption \ref{A1} and Assumption \ref{A2}, the solution $x^\dag$ of (\ref{sys}) satisfying (\ref{xdag}) is uniquely
defined.
\end{lemma}

\begin{proof}
This is essentially proved in \cite[Lemma 3.2]{jw13}.
\end{proof}

In order to construct an approximate solution to (\ref{sys}), we will formulate a two-point gradient method with
penalty term induced by the uniformly convex function $\Theta$. Let $\tau>1$ be a given number. By picking
$x_{-1}^\d = x_0^\d:=x_0\in\X$ and $\xi_{-1}^\d=\xi_0^\d:=\xi_0\in \partial\Theta(x_0)$ as initial guess, for $n\ge 0$ we define
\begin{equation} \label{TPGM}
\begin{aligned}
\zeta_n^\d &= \xi_{n}^\d +\lambda_n^\d (\xi_n^\d-\xi_{n-1}^\d),\\
z_n^\d &= \nabla \Theta^*(\zeta_n^\d),\\
\xi_{n+1}^\d &=\zeta_n^\d-\mu_{n}^\d L(z_n^\d)^* J_s^{\Y}(r_n^\d),\\
x_{n+1}^\d &= \nabla \Theta^*(\xi_{n+1}^\d),
\end{aligned}
\end{equation}
where $r_n^\d = F(z_n^\d)-y^\d$, $\lambda_n^\d \ge 0$ is the combination parameter, $\mu_n^\d$ is the step sizes defined by
\begin{align}\label{step}
\mu_n^\d= \left\{\begin{array}{lll}
\displaystyle{\min\left\{\frac{\bar \mu_0 \|r_n^\d\|^{p(s-1)}}{\|L(z_n^\d)^*J_s^{\Y}(r_n^\d)\|^p}, \bar \mu_1\right\}\|r_n^\d\|^{p-s}}
 & \mbox{ if } \|r_n^\d\|>\tau \d\\[2ex]
0 & \mbox{ if } \|r_n^\d\| \le \tau \d
\end{array}\right.
\end{align}
for some positive constants $\bar \mu_0$ and $\bar \mu_1$, and the mapping $J_s^{\Y}: \Y\rightarrow\Y^*$ with $1<s<\infty$ denotes the
duality mapping of $\Y$ with gauge function $t\rightarrow t^{s-1}$, which is single-valued and continuous because
$\Y$ is assumed to be uniformly smooth. We remark that when $\lambda_n^\d =0$ for all $n$, the method (\ref{TPGM}) reduces to the
Landweber iteration considered in \cite{jw13} where a detailed convergence analysis has been carried out. When $\lambda_n^\d = n/(n+\a)$
with $\a \ge 3$ for all $n$, the method (\ref{TPGM}) becomes a refined version of the Nesterov acceleration of Landweber iteration
proposed in \cite{j16}; although there is no convergence theory available, numerical simulations in \cite{j16} demonstrate its
usefulness and acceleration effect. In this paper we will consider (\ref{TPGM}) with $\lambda_n^\d$ satisfying suitable conditions
to be specified later. Note that our method (\ref{TPGM}) requires the use of the previous two iterations at every iteration step,
which follows the spirit from \cite{hr17}; on the other hand, our method allows the use of a general $p$-convex penalty
function $\Theta$, which could be non-smooth, to reconstruct solutions with special features such as sparsity and discontinuities.

\subsection{Convergence}

In order to use our TPG-$\Theta$ method (\ref{TPGM}) to produce a useful approximate solution to (\ref{sys}),
the iteration must be terminated properly. We will use the discrepancy principle with respect to $z_n^\d$,
i.e., for a given $\tau>1$, we will terminate the iteration after $n_\d$ steps, where $n_\d := n(\d,y^\d)$
is the integer such that
\begin{align}\label{dp}
\|F(z_{n_\d}^\d)-y^\d\| \le \tau\d < \|F(z_n^\d)-y^\d\|\,,\quad 0\le n<n_\d,
\end{align}
and use $x_{n_\d}^\d$ as an approximated solution. To carry out the convergence analysis of $x_{n_\d}^\d$ to
$x^\dag$ as $\d\rightarrow 0$, we are going to show that, for any solution $\hat x$ of \eqref{sys} in
$B_{2\rho}(x_0)\cap \D(\Theta)$, the Bregman distance $D_{\xi_n^\d}\Theta(\hat x, x_n^\d)$, $0\le n \le n_\d$, is
monotonically decreasing with respect to $n$. To this end, we introduce
\begin{align}\label{decrease0}
\triangle_n: = D_{\xi_n^\d}\Theta(\hat x, x_n^\d) - D_{\xi_{n-1}^\d}\Theta(\hat x, x_{n-1}^\d).
\end{align}
We will show that, under suitable choice of $\{\lambda_n^\d\}$, there holds $\Delta_n\le 0$ for $0\le n \le n_\d$.

\begin{lemma}\label{lem1}
Let $\X$ be reflexive, let $\Y$ be uniformly smooth, and let Assumption \ref{A1} and Assumption \ref{A2} hold.
Then, for any solution $\hat x$ of \eqref{sys} in  $B_{2\rho}(x_0)\bigcap \D(\Theta)$, there holds
\begin{align}
& D_{\zeta_n^\d}\Theta(\hat x, z_n^\d) -D_{\xi_{n}^\d}\Theta(\hat x, x_{n}^\d) \nonumber \\
& \qquad \qquad \le \lambda_n^\d\triangle_n+\frac{1}{p^*(2c_0)^{p^*-1}} \left(\lambda_n^\d+\left(\lambda_n^\d\right)^{p^*}\right)
\|\xi_{n}^\d-\xi_{n-1}^\d\|^{p^*}.\label{basic1}
\end{align}
If $z_n^\d\in B_{3\rho}(x_0)$ then
\begin{align}
D_{\xi_{n+1}^\d}\Theta(\hat x,x_{n+1}^\d) - D_{\zeta_n^\d}\Theta(\hat x,z_n^\d)
&\le -\left(1-\eta -\frac{1}{p^*} \left(\frac{\bar \mu_0}{2 c_0}\right)^{p^*-1}\right) \mu_n^\d \|F(z_n^\d)-y^\d\|^s\nonumber \\
& \quad \,  +(1+\eta)\mu_n^\d \|F(z_n^\d)-y^\d\|^{s-1} \d.\label{basic2}
\end{align}
\end{lemma}

\begin{proof}
By using the identity \eqref{2.1}, Lemma \ref{lem:2.1} and the definition of $\zeta_n^\d$, we can obtain
\begin{align*}
D_{\zeta_n^\d}\Theta(\hat x, z_n^\d) -D_{\xi_{n}^\d}\Theta(\hat x, x_{n}^\d)
&= D_{\zeta_n^\d}\Theta(x_n^\d, z_n^\d)+\l \zeta_n^\d- \xi_{n}^\d,x_n^\d- \hat{x}\r\\
&\le \frac{1}{p^*(2c_0)^{p^*-1}}\|\zeta_n^\d-\xi_n^\d\|^{p^*}+ \l \zeta_n^\d- \xi_{n}^\d,x_n^\d- \hat{x}\r\\
&= \frac{1}{p^*(2c_0)^{p^*-1}}(\lambda_n^\d)^{p^*}\|\xi_{n-1}^\d-\xi_n^\d\|^{p^*}+ \l \zeta_n^\d- \xi_{n}^\d,x_n^\d- \hat{x}\r\,.
\end{align*}
By using again the definition of $\zeta_n^\d$, \eqref{2.1} and Lemma \ref{lem:2.1}, and referring to
the definition of $\Delta_n$ and $\lambda_n^\d\ge 0$, we have
\begin{align*}
\l \zeta_n^\d- \xi_{n}^\d, x_n^\d- \hat{x}\r
&=\lambda_n^\d\l \xi_{n}^\d- \xi_{n-1}^\d,x_n^\d- \hat{x}\r\\
&=\lambda_n^\d
\left(D_{\xi_{n}^\d}\Theta(\hat x, x_{n}^\d)-D_{\xi_{n-1}^\d}\Theta(\hat x, x_{n-1}^\d)+D_{\xi_{n-1}^\d}\Theta(x_{n}^\d,x_{n-1}^\d)\right)\\
&\le \lambda_n^\d\triangle_n + \frac{1}{p^*(2c_0)^{p^*-1}}\lambda_n^\d \|\xi_{n}^\d-\xi_{n-1}^\d\|^{p^*}\,.
\end{align*}
The combination of the above two estimates yields \eqref{basic1}.

To derive \eqref{basic2}, we first use the identity \eqref{2.1} and Lemma \ref{lem:2.1} to obtain
\begin{align}\label{prepare}
& D_{\xi_{n+1}^\d}\Theta(\hat x,x_{n+1}^\d) - D_{\zeta_n^\d}\Theta(\hat x,z_n^\d) \nonumber \\
&\qquad = D_{\xi_{n+1}^\d}\Theta(z_n^\d,x_{n+1}^\d)+\l \xi_{n+1}^\d-\zeta_n^\d,z_n^\d-\hat x\r\nonumber\\
&\qquad \le \frac{1}{p^*(2c_0)^{p^*-1}}\|\xi_{n+1}^\d-\zeta_n^\d\|^{p^*} + \l \xi_{n+1}^\d-\zeta_n^\d,z_n^\d-\hat x\r.
\end{align}
Recall the definition of $\xi_{n+1}^\d$ in \eqref{TPGM}, we have
\begin{align*}
\|\xi_{n+1}^\d-\zeta_n^\d\|^{p^*}
&= (\mu_n^\d)^{p^*}\|L(z_n^\d)^*J_s^{\mathcal{Y}}(F(z_n^\d)-y^\d)\|^{p^*}.
\end{align*}
According to the definition (\ref{step}) of $\mu_n^\d$, one can see that
$$
\mu_n^\d \le \frac{\bar \mu_0 \|F(z_n^\d)-y^\d\|^{s(p-1)}}{\|L(z_n^\d)^* J_s^{\Y} (F(z_n^\d)-y^\d)\|^p},
$$
which implies that
\begin{align*}
(\mu_n^\d)^{p^*-1}\|L(z_n^\d)^*J_s^{\mathcal{Y}}(F(z_n^\d)-y^\d)\|^{p^*}
& \le \bar \mu_0^{p^*-1} \|F(z_n^\d)-y^\d\|^{s(p-1)(p^*-1)} \\
& = \bar \mu_0^{p^*-1} \|F(z_n^\d)-y^\d\|^s.
\end{align*}
Therefore, the first term on the right hand side of \eqref{prepare} can be estimated as
\begin{align}\label{11.3.1}
\frac{1}{p^*(2c_0)^{p^*-1}}\|\xi_{n+1}^\d-\zeta_n^\d\|^{p^*}
\le \frac{1}{p^*} \left(\frac{\bar \mu_0}{2c_0}\right)^{p^*-1} \mu_n^\d \|F(z_n^\d)-y^\d\|^s.
\end{align}
For the second term  on the right hand side of \eqref{prepare}, we may use the definition of $\xi_{n+1}^\d$,
the property of $J_s^{\Y}$, and the definition of $\mu_n^\d$ to derive that
\begin{align*}
\l \xi_{n+1}^\d-\zeta_n^\d,z_n^\d-\hat x\r
& = -\mu_n^\d\l J_s^{\mathcal{Y}}(F(z_n^\d)-y^\d),L(z_n^\d)(z_n^\d-\hat x)\r\\
& = -\mu_n^\d\l J_s^{\mathcal{Y}}(F(z_n^\d)-y^\d), y^\d-F(z_n^\d)-L(z_n^\d)(\hat x-z_n^\d)\r \\
& \quad \, - \mu_n^\d\|F(z_n^\d)-y^\d\|^s\\
& \le \mu_n^\d\|F(z_n^\d)-y^\d\|^{s-1}\left(\delta+\|y-F(z_n^\d)-L(z_n^\d) (\hat x-z_n^\d)\|\right) \\
& \quad \, -\mu_n^\d\|F(z_n^\d)-y^\d\|^s.
\end{align*}
Recall that $z_n^\d\in B_{3\rho}(x_0)$, we may use Assumption \ref{A2} (c) to further obtain
\begin{align}\label{11.3.2}
& \l \xi_{n+1}^\d-\zeta_n^\d,z_n^\d-\hat x\r \nonumber \\
& \le \mu_n^\d \|F(z_n^\d)-y^\d\|^{s-1}\left(\d + \eta\|F(z_n^\d)-y\|\right) -\mu_n^\d \|F(z_n^\d)-y^\d\|^s \nonumber \\
& \le (1+\eta) \mu_n^\d \|F(z_n^\d)-y^\d\|^{s-1} \d - (1-\eta) \mu_n^\d \|F(z_n^\d)-y^\d\|^s.
\end{align}
The combination of above two estimates (\ref{11.3.1}) and (\ref{11.3.2}) with (\ref{prepare}) yields \eqref{basic2}.
\end{proof}

\begin{lemma}\label{lem2}
Let $\X$ be reflexive, let $\Y$ be uniformly smooth, and let Assumption \ref{A1} and Assumption \ref{A2} hold.
Assume that $\tau>1$ and $\bar \mu_0>0$ are chosen such that
\begin{align}\label{c1}
c_1:=1-\eta-\frac{1+\eta}{\tau} -\frac{1}{p^*} \left(\frac{\bar\mu_0}{2 c_0}\right)^{p^*-1}>0.
\end{align}
If $z_n^\d\in B_{3\rho}(x_0)$, then, for any solution $\hat x$ of \eqref{sys} in $B_{2\rho}(x_0)\bigcap \D(\Theta)$, there holds
\begin{align}\label{decrease}
\triangle _{n+1} & \leq\lambda_n^\d\triangle_n+\frac{1}{p^*(2c_0)^{p^*-1}}\left(\lambda_n^\d
+(\lambda_n^\d)^{p^*}\right) \|\xi_{n}^\d-\xi_{n-1}^\d\|^{p^*} \nonumber \\
& \quad \, -c_1 \mu_n^\d \|F(z_n^\d)-y^\d\|^s,
\end{align}
where $\Delta_n$ is defined by (\ref{decrease0}).
\end{lemma}

\begin{proof}
By using the definition of $\mu_n^\d$ it is easily seen that $\mu_n^\d \d \le \mu_n^\d \|F(z_n^\d)-y^\d\|/\tau$.
It then follows from \eqref{basic2} that
\begin{align*}
D_{\xi_{n+1}^\d} \Theta(\hat x, x_{n+1}^\d) - D_{\zeta_n^\d}\Theta(\hat x, z_n^\d) \le -c_1 \mu_n^\d \|F(z_n^\d)-y^\d\|^s.
\end{align*}
Combining this estimate with \eqref{basic1} yields (\ref{decrease}).
\end{proof}

We will use Lemma \ref{lem1} and Lemma \ref{lem2} to show that $z_n^\d \in B_{3\rho}(x_0)$ and $\Delta_n \le 0$
for all $n\ge 0$ and that the integer $n_\d$ determined by (\ref{dp}) is finite. To this end, we need to place
conditions on $\{\lambda_n^\d\}$. We assume that $\{\lambda_n^\d\}$ is chosen such that
\begin{align}\label{add ass1}
\frac{1}{p^*(2c_0)^{p^*-1}} \left((\lambda_n^\d)^{p^*}+\lambda_n^\d\right)\|\xi_{n}^\d-\xi_{n-1}^\d\|^{p^*}\le c_0\rho^p
\end{align}
and
\begin{align}\label{condition}
\frac{1}{p^*(2c_0)^{p^*-1}}\left((\lambda_n^\d)^{p^*}+\lambda_n^\d\right) \|\xi_{n}^\d-\xi_{n-1}^\d\|^{p^*}
-\frac{c_1}{\nu} \mu_n^\d \|F(z_n^\d)-y^\d\|^s\le 0
\end{align}
for all $n\ge 0$, where $\nu>1$ is a constant independent of $\d$ and $n$. We will discuss how to choose $\{\lambda_n^\d\}$
to satisfy (\ref{add ass1}) and (\ref{condition}) shortly.

\begin{proposition}\label{monotonicity}
Let $\X$ be reflexive, let $\Y$ be uniformly smooth, and let Assumption \ref{A1} and Assumption \ref{A2} hold.
Let $\tau>1$ and $\bar \mu_0>0$ be chosen such that (\ref{c1}) holds.  If $\{\lambda_n^\d\}$ is chosen such that (\ref{add ass1}) and (\ref{condition}) hold, then
\begin{equation}\label{10.11.1}
z_n^\d \in B_{3\rho}(x_0) \quad \mbox{ and } \quad x_n^\d \in B_{2\rho}(x_0) \quad \mbox{for } n\ge 0.
\end{equation}
Moreover, for any solution $\hat{x}$ of \eqref{sys} in $B_{2\rho}(x_0)\bigcap \D(\Theta)$ there hold
\begin{align}\label{10.11.2}
D_{\xi_n^\d}\Theta(\hat x, x_n^\d)\leq D_{\xi_{n-1}^\d}\Theta(\hat x, x_{n-1}^\d)
\end{align}
and
\begin{align}\label{sum}
\sum_{m=0}^{n}\mu_m^\d \|F(z_m^\d)-y^\d\|^s \le \frac{\nu }{\left(\nu-1\right)c_1} D_{\xi_0}\Theta(\hat x, x_0)
\end{align}
for all $n\ge 0$.  Let $n_\d$ be chosen by the discrepancy principle \eqref{dp}, then $n_\d$ must be a finite integer.
\end{proposition}

\begin{proof}
We will show (\ref{10.11.1}) and (\ref{10.11.2}) by induction. Since $x_{-1}^\d = x_0^\d = x_0$, $\xi_{-1}^\d = \xi_0^\d = \xi_0$,
and $z_0^\d = \nabla \Theta^*(\zeta_0^\d) = \nabla \Theta^*(\xi_0) = x_0$, they are trivial for $n=0$. Now we assume that
(\ref{10.11.1}) and (\ref{10.11.2}) hold for all $0\le n\le m$ for some integer $m\ge 0$, we will show that they are also true
for $n = m+1$. By the induction hypotheses $z_m^\d \in B_{3\rho}(x_0)$, we may use Lemma \ref{lem2} and
(\ref{condition}) to derive that
\begin{align*}
\Delta_{m+1} \le \lambda_m^\d \Delta_m - \left(1-\frac{1}{\nu}\right) c_1\mu_m^\d \|F(z_m^\d)-y^\d\|^s.
\end{align*}
Since $\lambda_m^\d \ge 0$ and $\nu>1$, this together with the induction hypothesis $\Delta_m \le 0$ implies that
\begin{align}\label{10.11.3}
\Delta_{m+1} \le - \left(1-\frac{1}{\nu}\right) c_1 \mu_m^\d \|F(z_m^\d)-y^\d\|^s \le 0
\end{align}
which shows (\ref{10.11.2}) for $n=m+1$. Consequently, by taking $\hat x= x^\dag$ and using Assumption \ref{A2} (a), we have
$$
D_{\xi_{m+1}^\d} \Theta(x^\dag, x_{m+1}^\d) \le D_{\xi_m^\d}\Theta(x^\dag, x_m^\d) \le \cdots \le D_{\xi_0} \Theta(x^\dag, x_0)\le c_0 \rho^p.
$$
By virtue of Assumption \ref{A1}, we then have $c_0\|x_{m+1}^\d -x^\dag\|^p\le c_0 \rho^p$ which together with $x^\dag \in B_\rho(x_0)$
implies that $x_{m+1}^\d \in B_{2\rho}(x_0)$. Now we may use (\ref{basic1}) in Lemma \ref{lem1}, (\ref{add ass1}) and $\Delta_{m+1}\le 0$
to derive that
\begin{align*}
D_{\zeta_{m+1}^\d}\Theta(x^\dag, z_{m+1}^\d)
&\le D_{\xi_{m+1}^\d}\Theta(x^\dag, x_{m+1}^\d)+ \lambda_{m+1}^\d \triangle_{m+1} +c_0\rho^p \\
& \le D_{\xi_{m+1}^\d}\Theta(x^\dag, x_{m+1}^\d) +c_0\rho^p\\
&\le D_{\xi_{0}}\Theta(x^\dag, x_0)+c_0\rho^p \\
& \le 2c_0\rho^p\,.
\end{align*}
This together with Assumption \ref{A1} yields $\|x^\dag-z_{m+1}^\d\|\le 2^{1/p}\rho\le 2\rho$, and consequently
$z_{m+1}^\d\in B_{3\rho}(x_0)$. We therefore complete the proof of (\ref{10.11.1}) and (\ref{10.11.2}).

Since (\ref{10.11.1}) and (\ref{10.11.2}) are valid, the inequality (\ref{10.11.3}) holds for all $m\ge 0$.
Thus
\begin{align}\label{use later}
\left(1-\frac{1}{\nu}\right) c_1 \mu_m^\d \|F(z_m^\d)-y^\d\|^s
\le D_{\xi_{m}^\d}\Theta(\hat x, x_{m}^\d) - D_{\xi_{m+1}^\d}\Theta(\hat x, x_{m+1}^\d)
\end{align}
for $m\ge 0$. Hence, for any integer $n \ge 0$ we have
\begin{align}\label{10.11.4}
\left(1-\frac{1}{\nu}\right)c_1\sum_{m=0}^n \mu_m^\d \|y^\d-F(z_m^\d)\|^s
& \le D_{\xi_0}\Theta(\hat x, x_0)- D_{\xi_{n+1}^\d}\Theta(\hat x, x_{n+1}^\d) \nonumber \\
& \le D_{\xi_0}\Theta(\hat x, x_0)
\end{align}
which shows (\ref{sum}).

If $n_\d$ is not finite, then $\|F(z_m^\d)-y^\d\|>\tau \d$ for all integers $m$ and consequently, by using
$\|L(x)\| \le C_0$ from Assumption \ref{A2} (c) and the property of $J_s^{\Y}$, we have
\begin{align}\label{11.3.3}
\mu_m^\d &= \min\left\{ \frac{\bar \mu_0 \|F(z_m^\d)-y^\d\|^{p(s-1)}}{\|L(z_m^\d)^* J_s^{\Y}(F(z_m^\d)-y^\d)\|^p}, \bar \mu_1\right\}
\|F(z_m^\d)-y^\d\|^{p-s} \nonumber \\
& \ge \min\left\{ \frac{\bar \mu_0}{C_0^p}, \bar \mu_1\right\} \|F(z_m^\d)-y^\d\|^{p-s}.
\end{align}
Therefore, it follows from (\ref{sum}) that
\begin{align*}
\frac{\nu}{(\nu-1)c_1} D_{\xi_0} \Theta(\hat x, x_0)
&\ge \min\left\{\frac{\bar \mu_0}{C_0^p}, \bar \mu_1\right\} \sum_{m=0}^n \|F(z_m^\d)-y^\d\|^p\\
&\ge \min\left\{\frac{\bar \mu_0}{C_0^p}, \bar \mu_1\right\} (n+1) \tau^p \d^p
\end{align*}
for all $n\ge 0$. By taking $n\rightarrow \infty$ we derive a contradiction. Therefore, $n_\d$ must be finite.
\end{proof}

\begin{remark}\label{remark}
{\rm In the proof of Proposition \ref{monotonicity}, the condition (\ref{condition}) plays a crucial role. Note that, by the
definition of our method (\ref{TPGM}), $z_n^\d$ depends on $\lambda_n^\d$. Therefore, it is not immediately clear how to choose
$\lambda_n^\d$ to make (\ref{condition}) satisfied. One may ask if there exists $\lambda_n^\d$ such that \eqref{condition} holds.
Obviously $\lambda_n^\d=0$ satisfies the inequality, which correspond to the Landweber iteration. In order to achieve acceleration,
it is necessary to find nontrivial $\lambda_n^\d$. Note that when $\|F(z_n^\d)-y^\d\|\le \tau \d$ occurs, (\ref{condition})
forces $\lambda_n^\d =0$ because $\mu_n^\d =0$. Therefore we only need to consider the case $\|F(z_n^\d)-y^\d\| >\tau \d$. By using
(\ref{11.3.3}) we can derive a sufficient condition
\begin{align*}
\frac{1}{p^*(2c_0)^{p^*-1}} \left(\lambda_n^\d+(\lambda_n^\d)^{p^*}\right)\|\xi_{n}^\d-\xi_{n-1}^\d\|^{p^*}
\le M \tau^p \d^p,
\end{align*}
where
$$
M := \frac{c_1}{\nu} \min\left\{\frac{\bar \mu_0}{C_0^p}, \bar \mu_1\right\}.
$$
Considering the particular case when $p=2$, this thus leads to the choice
\begin{align}\label{choice lambda}
\lambda_n^\d: =\min\left\{-\frac{1}{2}+\sqrt{\frac{1}{4}+\frac{4c_0 M\tau^2\delta^2}{\|\xi_{n}^\d-\xi_{n-1}^\d\|^2}}, \frac{n}{n+\a}\right\}\,,
\end{align}
where $\a\ge 3$ is a given number. Note that in the above formula for $\lambda_n^\d$, inside the ``$\min$" the second argument is
taken to be $n/(n+\a)$ which is the combination parameter used in Nesterov's acceleration strategy; in case the first argument is large,
this formula may lead to $\lambda_n^\d = n/(n+\a)$ and consequently the acceleration effect of Nesterov can be utilized.
For general $p>1$, by placing the requirement $0\le \lambda_n^\d \le n/(n+\a) \le 1$, one may choose $\lambda_n^\d$ to satisfy
\begin{align*}
\frac{2 \lambda_n^\d}{p^*(2 c_0)^{p^*-1}} \|\xi_n^\d -\xi_{n-1}^\d\|^{p^*} \le M\tau^p \d^p
\end{align*}
which leads to the choice
\begin{align}\label{lambda2}
\lambda_n^\d = \min\left\{ \frac{\gamma_0 \d^p}{\|\xi_n^\d -\xi_{n-1}^\d\|^{p^*}}, \frac{n}{n+\a}\right\},
\quad \gamma_0 := \frac{1}{2} (2 c_0)^{p^*-1} p^* M\tau^p.
\end{align}
We remark that the choices of $\lambda_n^\d$ given in (\ref{choice lambda}) and (\ref{lambda2}) may decrease to 0 as $\d\rightarrow 0$,
consequently the acceleration effect could also decrease for $\d\rightarrow 0$. Since for small values of $\d$
the acceleration is needed most, other strategies should be explored. We will give a further consideration on the choice of
$\lambda_n^\d$ in the next subsection.
}
\end{remark}

In order to establish the regularization property of the method (\ref{TPGM}), we need to consider its noise-free counterpart.
By dropping the superscript $\d$ in all the quantities involved in (\ref{TPGM}), it leads to the following formulation
of the two-point gradient method for the noise-free case:
\begin{equation} \label{TPGM_E}
\begin{aligned}
\zeta_n &= \xi_{n} +\lambda_n (\xi_n-\xi_{n-1}),\\
z_n &= \nabla \Theta^*(\zeta_n),\\
\xi_{n+1} &= \zeta_n-\mu_{n} L(z_n)^* J_s^{\Y}(r_n),\\
x_{n+1} &= \nabla \Theta^*(\xi_{n+1}),
\end{aligned}
\end{equation}
where $r_n = F(z_n)-y$, $\lambda_n\ge 0$ is the combination parameter, and $\mu_n$ is the step size given by
\begin{align*}
\mu_n = \left\{\begin{array}{lll}
\displaystyle{\min\left\{\frac{\bar \mu_0 \|r_n\|^{p(s-1)}}{\|L(z_n)^*J_s^{\Y}(r_n)\|^p}, \bar \mu_1\right\}} \|r_n\|^{p-s}
  & \mbox{ if } F(z_n) \ne y,\\[2ex]
0 & \mbox{ if } F(z_n) = y.
\end{array}\right.
\end{align*}
We will first establish a convergence result for (\ref{TPGM_E}). The following result plays a crucial role
in the argument.

\begin{proposition}\label{general}
Consider the equation (\ref{sys}) for which Assumption \ref{A2} holds.
Let $\Theta: \X\to (-\infty, \infty]$ be a proper, lower semi-continuous and uniformly convex function.
Let $\{x_n\}\subset B_{2\rho}(x_0)$ and $\{\xi_n\}\subset \X^*$ be such that

\begin{enumerate}
\item[(i)] $\xi_n\in \p \Theta(x_n)$ for all $n$;

\item[(ii)] for any solution $\hat x$ of (\ref{sys}) in $B_{2\rho}(x_0)\cap \D(\Theta)$ the sequence
$\{D_{\xi_n}\Theta(\hat x, x_n)\}$ is monotonically decreasing;

\item[(iii)] $\lim_{n\rightarrow \infty} \|F(x_n)-y\|=0$.

\item[(iv)] there is a subsequence $\{n_k\}$ with $n_k\rightarrow \infty$ such that for any solution
$\hat x$ of (\ref{sys}) in $B_{2\rho}(x_0)\cap \D(\Theta)$ there holds
\begin{equation}\label{12.15.11}
\lim_{l\rightarrow \infty} \sup_{k\ge l} |\l \xi_{n_k}-\xi_{n_l}, x_{n_k}-\hat x\r| = 0.
\end{equation}
\end{enumerate}
Then there exists a solution $x_*$ of (\ref{sys}) in $B_{2\rho}(x_0)\cap \D(\Theta)$ such that
\begin{equation*}
\lim_{n\rightarrow \infty} D_{\xi_n}\Theta(x_*, x_n)=0.
\end{equation*}
If, in addition, $x^\dag \in B_\rho(x_0)\cap \D(\Theta)$ and $\xi_{n+1}-\xi_n \in \overline{{\mathcal R}(L(x^\dag)^*)}$
for all $n$, then $x_*=x^\dag$.
\end{proposition}

\begin{proof}
This result essentially follows from \cite[Proposition 3.6]{jw13} and its proof.
\end{proof}

\begin{theorem}\label{th noisefree}
Let $\X$ be reflexive, let $\Y$ be uniformly smooth, and let Assumption \ref{A1} and Assumption \ref{A2} hold.
Assume that $\bar \mu_0>0$ is chosen such that
$$
1-\eta - \frac{1}{p^*} \left(\frac{\bar \mu_0}{2 c_0}\right)^{p^*-1} >0
$$
and the combination parameters $\{\lambda_n\}$ are chosen to satisfy the counterparts of \eqref{add ass1}
and \eqref{condition} with $\d = 0$  and
\begin{align}\label{add ass3}
\sum_{n=0}^\infty\lambda_n\|\xi_n-\xi_{n-1}\|<\infty.
\end{align}
Then, there exists a solution $x_*$ of (\ref{sys}) in $B_{2\rho}(x_0)\cap \D(\Theta)$ such that
\begin{equation*}
\lim_{n\rightarrow \infty} \|x_n-x_*\|=0\quad  \mbox{ and } \quad \lim_{n\rightarrow \infty} D_{\xi_n}\Theta(x_*, x_n)=0.
\end{equation*}
If in addition $\N(L(x^\dag))\subset \N(L(x))$ for all $x\in B_{3\rho}(x_0)$, then $x_*=x^\dag$.
\end{theorem}

\begin{proof} We will use Proposition \ref{general} to prove the result. By the definition $x_n = \nabla \Theta^*(\xi_n)$
we have $\xi_n \in \p \Theta(x_n)$ which shows (i) in Proposition \ref{general}. By using the same argument for
proving Proposition \ref{monotonicity} we can show that $z_n \in B_{3\rho}(x_0)$ and $x_n \in B_{2\rho}(x_0)$ for all $n$ with
\begin{align}\label{2018.11.1}
D_{\xi_{n+1}}\Theta(\hat x, x_{n+1}) \le D_{\xi_n} \Theta(\hat x, x_n)
\end{align}
for any solution $\hat x$ of (\ref{sys}) in $B_{2\rho}(x_0) \cap \D(\Theta)$ and
\begin{align}\label{2018.11.1.1}
\sum_{n=0}^\infty \mu_n \|F(z_n)-y\|^s <\infty.
\end{align}
From (\ref{2018.11.1}) it follows that (ii) in Proposition \ref{general} holds. Moreover, by using the definition of
$\mu_n$ and the similar derivation for (\ref{11.3.3}) we have
$$
\min\left\{\frac{\bar \mu_0}{C_0^p}, \bar \mu_1\right\} \|F(z_n)-y\|^p
\le \mu_n \|F(z_n)-y\|^s \le \bar \mu_1 \|F(z_n)-y\|^p.
$$
Thus it follows from (\ref{2018.11.1.1}) that
$$
\sum_{n=0}^\infty \|F(z_n)-y\|^p <\infty.
$$
Consequently
\begin{align}\label{residual limit}
\lim_{n\rightarrow \infty} \|F(z_n)-y\|= 0.
\end{align}
By using Assumption \ref{A2} (c), (\ref{TPGM_E}), (\ref{2.4}) and (\ref{condition}) with $\d =0$, we have
\begin{align}\label{10.12.5}
\|F(x_n)-F(z_n)\|
& \le \frac{1}{1-\eta} \|L(z_n) (x_n-z_n)\| \le \frac{C_0}{1-\eta} \|x_n-z_n\|\nonumber \displaybreak[0]\\
& = \frac{C_0}{1-\eta} \|\nabla\Theta^*(\xi_n)-\nabla \Theta^*(\zeta_n)\| \nonumber \\
& \le \frac{C_0}{(1-\eta)(2c_0)^{p^*-1}}\|\xi_n-\zeta_n\|^{p^*-1} \nonumber \displaybreak[0]\\
& = \frac{C_0}{(1-\eta)(2c_0)^{p^*-1}}\lambda_n^{p^*-1}\|\xi_n-\xi_{n-1}\|^{p^*-1} \nonumber \displaybreak[0]\\
& \le \frac{C_0}{1-\eta} \left(\frac{c_1 p^*}{2 c_0 \nu}\right)^{1/p} (\mu_n\|F(z_n)-y\|^s)^{1/p} \nonumber \displaybreak[0]\\
& \le \frac{C_0}{1-\eta} \left(\frac{c_1 p^*\bar \mu_1}{2 c_0 \nu}\right)^{1/p} \|F(z_n)-y\|.
\end{align}
The combination of (\ref{residual limit}) and (\ref{10.12.5}) implies that $\|F(x_n)-y\| \rightarrow 0$ as $n\rightarrow \infty$
which shows (iii) in Proposition \ref{general}.

In order to establish the convergence result, it remains only to show (iv) in Proposition \ref{general}.
To this end, we consider $\|F(z_n)-y\|$. It is known that $\|F(z_n)-y\| \rightarrow 0$ as $n\rightarrow \infty$.
If $\|F(z_n)-y\| =0$ for some $n$, then (\ref{condition}) with $\d =0$ forces $\lambda_n (\xi_n-\xi_{n-1}) =0$.
Thus $\zeta_n = \xi_n$. On the other hand, we also have $\mu_n=0$ and hence $\xi_{n+1} = \zeta_n$. Consequently
$\xi_{n+1} = \zeta_n =\xi_n$ and
$$
\zeta_{n+1} = \xi_{n+1} + \lambda_{n+1}(\xi_{n+1}-\xi_n) = \xi_{n+1} = \zeta_n.
$$
Thus
$$
z_{n+1} = \nabla \Theta^*(\zeta_{n+1}) = \nabla \Theta^*(\zeta_n) = z_n
$$
which implies that $F(z_{n+1})= F(z_n) =y$. By repeating the argument one can see
that $F(z_m)=y$ for all $m\ge n$. Therefore we can choose a strictly increasing sequence
$\{n_k\}$ of integers by letting $n_0=0$ and for each $k\ge 1$, letting $n_k$ be the first integer satisfying
\begin{align*}
n_k\geq n_{k-1}+1\quad\textrm{ and }\quad  \|F(z_{n_k})-y\| \le \|F(z_{n_{k-1}})-y\|.
\end{align*}
For such chosen strictly increasing sequence $\{n_k\}$, it is easily seen that
\begin{align}\label{10.11.5}
\|F(z_{n_k})-y\| \leq \|F(z_n)-y\|,\quad 0\leq n<n_k.
\end{align}
For any integers $0\le l<k<\infty$, we consider
$$
\l\xi_{n_k}-\xi_{n_l},x_{n_k}-\hat x\r =\sum_{n=n_l}^{n_k-1} \l \xi_{n+1}-\xi_n, x_{n_k} -\hat x\r.
$$
By using the definition of $\xi_{n+1}$ we have
$$
\xi_{n+1}-\xi_n = \lambda_n (\xi_n-\xi_{n-1}) -\mu_n L(z_n)^* J_s^{\Y}(F(z_n)-y).
$$
Therefore, by using the property of $J_s^{\Y}$, we have
\begin{align}\label{10.12.1}
\left|\l\xi_{n_k}-\xi_{n_l},x_{n_k}-\hat x\r\right|
& \le \sum_{n=n_l}^{n_k-1} \lambda_n |\l \xi_n-\xi_{n-1}, x_{n_k}-\hat x\r| \nonumber\\
& \quad \, + \sum_{n=n_l}^{n_k-1} \mu_n |\l J_s^{\Y} (F(z_n)-y), L(z_n) (x_{n_k}-\hat x)\r| \nonumber \displaybreak[0]\\
&\le \sum_{n=n_l}^{n_k-1}\lambda_n\|\xi_{n}-\xi_{n-1}\|\|x_{n_k}-\hat x\| \nonumber \\
& \quad \, + \sum_{n=n_l}^{n_k-1}\mu_n\|F(z_n)-y\|^{s-1}\|L(z_n)(x_{n_k}-\hat x)\|.
\end{align}
By using Assumption \ref{A2} (c) and (\ref{10.11.5}), we obtain for $n<n_k$ that
\begin{align*}
\|L(z_n)(x_{n_k}-\hat x)\|&\le \|L(z_n)(x_{n_k}-z_n)\|+\|L(z_n)(z_n-\hat x)\|\\
&\le (1+\eta)\left(\|F(x_{n_k})-F(z_n)\| + \|F(z_n)-y\|\right)\\
&\le 2(1+\eta)\|F(z_n)-y\| + (1+\eta)\|F(x_{n_k})-y\|\\
&\le 2(1+\eta)\|F(z_n)-y\| \\
& \quad \, + (1+\eta)\left(\|F(x_{n_k})-F(z_{n_k})\|+\|F(z_{n_k})-y\|\right)\\
& \le 3(1+\eta) \|F(z_n)-y\| + (1+\eta) \|F(x_{n_k})-F(z_{n_k})\|.
\end{align*}
By using (\ref{10.12.5}) and (\ref{10.11.5}), we have for $n<n_k$ that
\begin{align*}
\|F(x_{n_k})-F(z_{n_k})\|
& \le \frac{C_0}{1-\eta} \left(\frac{c_1 p^*\bar \mu_1}{2 c_0 \nu}\right)^{1/p} \|F(z_{n_k})-y\|\\
& \le \frac{C_0}{1-\eta} \left(\frac{c_1 p^*\bar \mu_1}{2 c_0 \nu}\right)^{1/p} \|F(z_n)-y\|.
\end{align*}
Therefore
\begin{align}\label{10.12.2}
\|L(z_n)(x_{n_k}-\hat x)\| \le C_1 \|F(z_n)-y\|
\end{align}
for $n<n_k$, where $C_1 := 3(1+\eta) + \frac{(1+\eta)C_0}{1-\eta} \left(\frac{c_1 p^*\bar \mu_1}{2 c_0 \nu}\right)^{1/p}$.
Combining (\ref{10.12.2})
with (\ref{10.12.1}) and using $x_{n_k}\in B_{2\rho}(x_0)$ we obtain
\begin{align*}
\left|\l\xi_{n_k}-\xi_{n_l},x_{n_k}-\hat x\r\right|
& \le \sum_{n=n_l}^{n_k-1}\lambda_n\|\xi_{n}-\xi_{n-1}\|\|x_{n_k}-\hat x\|+ C_1 \sum_{n=n_l}^{n_k-1} \mu_n \|F(z_n)-y\|^s\\
& \le 4\rho \sum_{n=n_l}^{n_k-1}\lambda_n\|\xi_{n}-\xi_{n-1}\|+ C_1 \sum_{n=n_l}^{n_k-1} \mu_n \|F(z_n)-y\|^s.
\end{align*}
By making use of \eqref{use later}, we obtain, with $C_2:= \nu C_1/((\nu-1) c_1)$,
\begin{align}\label{10.12.7}
& \left|\l\xi_{n_k}-\xi_{n_l},x_{n_k}-\hat x\r\right| \nonumber \\
&\le 4 \rho\sum_{n=n_l}^{n_k-1}\lambda_n\|\xi_{n}-\xi_{n-1}\|
+ C_2 \sum_{n=n_l}^{n_k-1}\left(D_{\xi_n}\Theta(\hat x,x_n)-D_{\xi_{n+1}}\Theta(\hat x,x_{n+1})\right) \nonumber \displaybreak[0]\\
& = 4 \rho \sum_{n=n_l}^{n_k-1}\lambda_n\|\xi_{n}-\xi_{n-1}\|
+C_2 \left(D_{\xi_{n_l}}\Theta(\hat x,x_{n_l})-D_{\xi_{n_k}}\Theta(\hat x,x_{n_k})\right).
\end{align}
Let $\gamma := \lim_{n\rightarrow \infty} D_{\xi_n}\Theta(\hat x, x_n)$ whose existence is guaranteed by the monotonicity
of $\{D_{\xi_n}\Theta(\hat x, x_n)\}$. Then
\begin{align*}
\sup_{k\ge l} \left|\l\xi_{n_k}-\xi_{n_l},x_{n_k}-\hat x\r\right|
& \le 4 \rho \sum_{n=n_l}^\infty\lambda_n\|\xi_{n}-\xi_{n-1}\| + C_2 \left(D_{\xi_{n_l}}\Theta(\hat x,x_{n_l})-\gamma\right).
\end{align*}
Thus it follows from (\ref{add ass3}) that
\begin{align*}
\lim_{l\rightarrow \infty} \sup_{k\ge l} \left|\l\xi_{n_k}-\xi_{n_l},x_{n_k}-\hat x\r\right|
\le C_2 \left(\lim_{l\rightarrow \infty} D_{\xi_{n_l}}\Theta(\hat x,x_{n_l})-\gamma\right)=0
\end{align*}
which verifies (iv) in Proposition \ref{general}.

To show $x_* = x^\dag$ under the additional condition $\N(L(x^\dag))\subset \N(L(x))$ for
all $x\in B_{3\rho}(x_0)$, we observe from (\ref{TPGM_E}) and $\xi_0-\xi_{-1} =0$ that
\begin{equation*}\label{n}
  \begin{aligned}
\xi_{n+1}-\xi_n &= -\mu_n L(z_n)^*(F(z_n)-y) + \lambda_n(\xi_n-\xi_{n-1})\\
& = -\sum_{k=0}^{n}\left(\prod_{i=k+1}^{n} \lambda_i \right)\mu_k L(z_k)^*(F(z_k)-y).
\end{aligned}
\end{equation*}
Since $\X$ is reflexive and $\N(L(x^\dag))\subset \N(L(x))$, we have $\overline{\R(L(x)^*)}\subset
\overline{\R(L(x^\dag)^*)}$ for all $x\in B_{3\rho}(x_0)$. Recall that $z_k\in B_{3\rho}(x_0)$. It thus
follows from the above formula that $\xi_{n+1}-\xi_n \in \overline{{\mathcal R}(L(x^\dag)^*)}$. Therefore
we may use the second part of Proposition \ref{general} to conclude the proof.
\end{proof}

Next, we are going to show that, using the discrepancy principle \eqref{dp} as a stopping rule, our method (\ref{TPGM})
becomes a convergent regularization method, if we additionally assume that $\lambda_n^\d$ depends continuously on $\d$
in the sense that $\lambda_n^\d \rightarrow \lambda_n$ as $\d\rightarrow 0$ for all $n$. We need the following stability result.

\begin{lemma}\label{stability}
Let $\X$ be reflexive, let $\Y$ be uniformly smooth, and let Assumption \ref{A1} and Assumption \ref{A2} hold.
Assume that $\tau>1$ and $\bar \mu_0>0$ are chosen to satisfy (\ref{c1}).  Assume also that the combination parameters
$\{\lambda_n^\d\}$ are chosen to depend continuously on $\d$ as $\d \rightarrow 0$ and satisfy \eqref{add ass1},
\eqref{condition} and \eqref{add ass3}. Then for all $n\ge 0$ there hold
\begin{equation*}
\zeta_n^\d \rightarrow \zeta_n, \quad z_n^\d \rightarrow z_n, \quad \xi_{n}^\d \rightarrow \xi_{n}
\quad \mbox{and} \quad x_{n}^\d\rightarrow x_{n} \quad \textrm{as } \d\rightarrow 0.
\end{equation*}
\end{lemma}

\begin{proof}
The result is trivial for $n=0$. We next assume that the result is true for all $0\leq n\le m$ and show that
the result is also true for $n = m+1$. We consider two cases.

{\it Case 1: $F(z_m)=y$}. In this case we have $\mu_m=0$ and $\|F(z_m^\d) -y^\d\|\rightarrow 0$
as $\d\rightarrow 0$ by the continuity of $F$ and the induction hypothesis $z_m^\d \rightarrow z_m$. Thus
\begin{equation*}
\xi_{m+1}^\d-\xi_{m+1} =\zeta_m^\d-\zeta_m -\mu_m^\d L(z_m^\d)^* J_s^{\Y} (F(z_m^\d)-y^\d),
\end{equation*}
which together with the definition of $\mu_m^\d$ and the induction hypothesis $\zeta_m^\d \rightarrow \zeta_m$
implies that
\begin{equation*}
\|\xi_{m+1}^\d-\xi_{m+1}\| \le\|\zeta_m^\d-\zeta_m\| + C_0 \bar\mu_1 \|F(z_m^\d)-y^\d\|^{p-1}\rightarrow 0
\quad \mbox{ as } \d\rightarrow 0.
\end{equation*}
Consequently, by using the continuity of $\nabla \Theta^*$ we have $x_{m+1}^\d =\nabla \Theta^*(\xi_{m+1}^\d)\rightarrow
\nabla \Theta^*(\xi_{m+1}) =x_{m+1}$ as $\d\rightarrow 0$. Recall that
$$
\zeta_{m+1}^\d = \xi_{m+1}^\d +\lambda_{m+1}^\d (\xi_{m+1}^\d-\xi_m^\d), \qquad
z_{m+1}^\d = \nabla \Theta^*(\zeta_{m+1}^\d).
$$
We may use the condition $\lambda_{m+1}^\d \rightarrow \lambda_{m+1}$ to conclude that $\zeta_{m+1}^\d\rightarrow \zeta_{m+1}$
and $z_{m+1}^\d\rightarrow z_{m+1}$ as $\d \rightarrow 0$.

{\it Case 2: $F(z_m)\ne y$}. In this case we have $\|F(z_m^\d)-y^\d\|\geq\tau \d$ for small $\d>0$. Therefore
\begin{align*}
\mu_m^\d &= \min\left\{\frac{\bar \mu_0 \|F(z_m^\d)-y^\d\|^{p(s-1)}}{\|L(z_m^\d)^*J_s^{\Y} (F(z_m^\d)-y^\d)\|^p}, \bar\mu_1\right\} \|F(z_m^\d)-y^\d\|^{p-s},\\
\mu_m &= \min\left\{\frac{\bar \mu_0 \|F(z_m)-y\|^{p(s-1)}}{\|L(z_m)^*J_s^{\Y} (F(z_m)-y)\|^p}, \bar\mu_1\right\} \|F(z_m)-y\|^{p-s}.
\end{align*}
If $L(z_m)^*J_s^{\Y}(F(z_m)-y)\ne 0$, then, by the induction hypothesis on $z_m^\d$, it is easily seen that
$\mu_m^\d \rightarrow \mu_m$ as $\d\rightarrow 0$.
If $L(z_m)^*J_s^{\Y}(F(z_m)-y) = 0$, then $\mu_m= \bar \mu_1 \|F(z_m)-y\|^{p-s}$ and $\mu_m^\d = \bar \mu_1 \|F(z_m^\d)-y^\d\|^{p-s}$
for small $\d>0$. This again implies that $\mu_m^\d \rightarrow \mu_m$ as $\d\rightarrow 0$.
Consequently, by utilizing the continuity of $F$, $L$, $J_s^{\Y}$ and $\nabla \Theta^*$ and the induction hypotheses,
we can conclude that $\xi_{m+1}^\d\rightarrow \xi_{m+1}$, $x_{m+1}^\d\rightarrow x_{m+1}$, $\zeta_{m+1}^\d \rightarrow \zeta_{m+1}$
and $z_{m+1}^\d \rightarrow z_{m+1}$ as $\d\rightarrow 0$.
\end{proof}

We are now in a position to give the main convergence result on our method (\ref{TPGM}).

\begin{theorem}\label{T6.4}
Let $\X$ be reflexive, let $\Y$ be uniformly smooth, and let Assumption \ref{A1} and Assumption \ref{A2} hold.
Assume that $\tau>1$ and $\bar \mu_0>0$ are chosen to satisfy (\ref{c1}).  Assume also that the combination parameters
$\{\lambda_n^\d\}$ are chosen to depend continuously on $\d$ as $\d \rightarrow 0$ and satisfy \eqref{add ass1},
\eqref{condition} and \eqref{add ass3}. Let $n_\d$ be chosen according to the discrepancy principle \eqref{dp}.
Then there exists a solution $x_*$ of (\ref{sys}) in $B_{2\rho}(x_0)\cap \D(\Theta)$ such that
\begin{equation*}
\lim_{\d\rightarrow 0} \|x_{n_\d}^\d-x_*\|=0 \qquad \mbox{and} \qquad
\lim_{\d\rightarrow 0} D_{\xi_{n_\d}^\d}\Theta(x_*, x_{n_\d}^\d) =0.
\end{equation*}
If in addition $\N(L(x^\dag))\subset \N(L(x))$ for all $x\in B_{3\rho}(x_0)$, then $x_*=x^\dag$.
\end{theorem}
\begin{proof}
Let $x_*$ be the solution of (\ref{sys}) determined in Theorem \ref{th noisefree}.
We complete the proof by considering two cases.

{\it Case 1}: Assume first that $\{y^{\d_l}\}$ is a sequence satisfying $\|y^{\d_l}-y\|\le \d_l$
with $\d_l\rightarrow 0$ such that $n_l:=n_{\d_l}\rightarrow \hat n$ as $l\rightarrow \infty$ for some finite
integer $\hat n$. We may assume $n_l = \hat n$ for all $l$. According to Lemma \ref{stability}, we have
$$
\zeta_{n_l}^{\d_l} \rightarrow \zeta_{\hat n}, \quad z_{n_l}^{\d_l} \rightarrow z_{\hat n}, \quad
\xi_{n_l}^{\d_l} \rightarrow \xi_{\hat n} \quad \mbox{ and } \quad x_{n_l}^{\d_l} \rightarrow x_{\hat n} \quad
\mbox{as } l\rightarrow \infty.
$$
By the definition of $n_l: = n_{\d_l}$, we have
\begin{equation*}
\|F(z_{n_l}^{\d_l})-y^{\d_l}\|\le \tau \d_l.
\end{equation*}
By using the similar argument for deriving (\ref{10.12.5}) we have
$$
\|F(x_{n_l}^{\d_l})-F(z_{n_l}^{\d_l})\| \le C \|F(z_{n_l}^{\d_l})-y^{\d_l}\|
$$
for some universal constant $C$. Thus
$$
\|F(x_{n_l}^{\d_l})-y^{\d_l}\| \le (1+C) \|F(z_{n_l}^{\d_l})-y^{\d_l}\| \le (1+C) \tau \d_l.
$$
Taking $l\rightarrow \infty$ and using the continuity of $F$ gives $F(x_{\hat n}) =y$. Thus $x_{\hat n}$ is a solution
of (\ref{sys}) in $B_{2\rho}(x_0)\cap \D(\Theta)$. By the monotonicity of $\{D_{\xi_n}\Theta(x_{\hat n}, x_n)\}$
with respect to $n$, we then obtain
$$
D_{\xi_n}\Theta(x_{\hat n}, x_n) \le D_{\xi_{\hat n}}\Theta(x_{\hat n}, x_{\hat n})=0, \quad \forall n\ge \hat n.
$$
Therefore $x_n = x_{\hat n}$ for all $n\ge \hat n$. Since Theorem \ref{th noisefree} shows that $x_n \rightarrow x^*$ as $n\rightarrow \infty$,
we must have $x_{\hat n} = x^*$ and thus $x_{n_l}^{\d_l} \rightarrow x^*$ as $l\rightarrow \infty$.
This together with the lower semi-continuity of $\Theta$ shows that
\begin{align*}
0 & \le \liminf_{l\rightarrow \infty} D_{\xi_{n_l}^{\d_l}}\Theta(x_*, x_{n_l}^{\d_l})
\le \limsup_{k\rightarrow \infty} D_{\xi_{n_l}^{\d_l}}\Theta(x_*, x_{n_l}^{\d_l})\\
& \leq \Theta(x_*) -\liminf_{l\rightarrow \infty} \Theta(x_{n_l}^{\d_l})
- \lim_{l\rightarrow \infty} \l \xi_{n_l}^{\d_l}, x_*-x_{n_l}^{\d_l}\r\\
& \le \Theta(x_*)-\Theta(x_*)=0,.
\end{align*}
Therefore  $\lim_{l\rightarrow \infty} D_{\xi_{n_l}^{\d_l}}\Theta(x_*, x_{n_l}^{\d_l})=0$.

{\it Case 2}: Assume $\{y^{\d_l}\}$ is a sequence satisfying $\|y^{\d_l}-y\|\le \d_l$
with $\d_l\rightarrow 0$ such that $n_l:=n_{\d_l}\rightarrow \infty$ as $l\rightarrow \infty$.
Let $n$ be any fixed integer, then $n_l>n$ for large $l$. It then follows from Lemma \ref{monotonicity} that
\begin{equation*}
D_{\xi_{n_l}^{\d_l}} \Theta(x_*, x_{n_l}^{\d_l}) \le D_{\xi_n^{\d_l}} \Theta (x_*, x_n^{\d_l})
=\Theta(x_*) -\Theta(x_n^{\d_l}) -\l \xi_n^{\d_l}, x_*-x_n^{\d_l}\r.
\end{equation*}
By using Lemma \ref{stability} and the lower semi-continuity of $\Theta$ we obtain
\begin{equation*}\label{n}
\begin{aligned}
0 & \le \liminf_{l\rightarrow \infty} D_{\xi_{n_l}^{\d_l}} \Theta (x_*, x_{n_l}^{\d_l})
\le \limsup_{l\rightarrow \infty} D_{\xi_{n_l}^{\d_l}} \Theta (x_*, x_{n_l}^{\d_l})\\
& \le \Theta(x_*) -\liminf_{l\rightarrow \infty} \Theta(x_n^{\d_l}) -\lim_{l\rightarrow \infty} \l \xi_n^{\d_l}, x_*-x_n^{\d_l}\r \\
& \le \Theta(x_*) -\Theta(x_n) -\l \xi_n, x_*-x_n\r\\
&= D_{\xi_n}\Theta(x_*, x_n)\,.
\end{aligned}
\end{equation*}
Since $n$ can be arbitrary and Theorem \ref{th noisefree} implies that $D_{\xi_n}\Theta(x_*, x_n)\rightarrow 0$
as $n\rightarrow \infty$, we therefore have $\lim_{l\rightarrow \infty} D_{\xi_{n_l}^{\d_l}} \Theta (x_*, x_{n_l}^{\d_l})=0$
and hence $\lim_{l\rightarrow \infty} \|x_{n_l}^{\d_l} - x_*\| =0$.
\end{proof}

\subsection{DBTS: the choice of $\lambda_n^\d$}

In this section we will discuss the choice of the combination parameter $\lambda_n^\d$ which leads to a convergent
regularization method.

In Remark \ref{remark} we have briefly discussed how to choose the combination parameter leading to the formulae
(\ref{choice lambda}) and (\ref{lambda2}). However, these choices of $\lambda_n^\d$ decrease to $0$ as $\d \rightarrow 0$,
and consequently the acceleration effect will decrease as $\d\rightarrow 0$ as well. Therefore, it is necessary to
find out other strategy for generating $\lambda_n^\d$ such that (\ref{condition}) and (\ref{add ass3}) hold.
We will adapt the discrete backtracking search (DBTS) algorithm introduced in \cite{hr17} to our situation.
To this end, we take a function $q:\mathbb{N}\to \mathbb{N}$ that is non-increasing and
\begin{align}\label{qn}
\sum_{i=0}^\infty q(i)<\infty.
\end{align}
The DBTS algorithm for choosing the combination parameter $\lambda_n^\d$ in our method (\ref{TPGM})
is formulated in Algorithm \ref{DBTSAlgorithm} below. Comparing with the one in \cite{hr17}, there are two modifications:
The first modification is the definition of $\beta_n$ in which we place $\beta_n(i) \le n/(n+\a)$ instead of $\beta_n(i) \le 1$;
this modification gives the possibility to speed up convergence by making use of the Nesterov's acceleration strategy. The second modification
is in the ``\textbf{Else}" part, where instead of setting $\lambda_n^\d =0$ we calculate $\lambda_n^\d$ by (\ref{lambda2}); this
modification can provide additional acceleration to speed up convergence.

\begin{algorithm}
\begin{algorithmic}
\STATE   \textbf{Given} $\xi_{n}^\d$, $\xi_{n-1}^\d$, $\tau$, $\delta$, $c_1$, $\nu$, $q:\mathbb{N}\rightarrow\mathbb{N}$,
$i_{n-1}^\d\in\mathbb{N}$, $j_{\max}\in\mathbb{N}$
\STATE  \textbf{Set } $\gamma_1 = c_1 p^* (2c_0)^{p^*-1}/\nu$
\STATE \textbf{Calculate} $\|\xi_n^\d-\xi_{n-1}^\d\|$ and define, with $\a\ge 3$,
\begin{align*}
\beta_n(i) = \min\left\{\frac{q(i)}{\|\xi_n^\d-\xi_{n-1}^\d\|}, \frac{p^*(2 c_0)^{p^*}\rho^p}{4 \|\xi_n^\d - \xi_{n-1}^\d\|^{p^*}}, \frac{n}{n+\a}\right\}.
\end{align*}
\STATE \textbf{For} $j=1,\ldots, j_{\max}$
\STATE    \qquad Set $\lambda_n^\d = \beta_n(i_{n-1}^\d+j)$;
\STATE    \qquad Calculate $\zeta_n^\d = \xi_{n}^\d +\lambda_n^\d (\xi_n^\d-\xi_{n-1}^\d)$ and $z_n^\d = \nabla \Theta^*(\zeta_n^\d)$;
\STATE    \qquad Calculate $\mu_n^\d$ by \eqref{step};
\STATE    \qquad \textbf{If} $\|y^\d-F(z_n^\d)\| \le \tau\d$
\STATE    \qquad \qquad $\lambda_n^\d = 0$;
\STATE    \qquad\qquad $i_n^\d = i_{n-1}^\d+j$;
\STATE    \qquad\qquad\textbf{break};
\STATE    \qquad\textbf{Else if} $(\lambda_n^\d+\left(\lambda_n^\d\right)^{p^*}) \|\xi_{n}^\d-\xi_{n-1}^\d\|^{p^*}\le \gamma_1 \mu_n^\d \|F(z_n^\d)-y^\d\|^s$
\STATE    \qquad\qquad $i_n^\d = i_{n-1}^\d + j$;
\STATE    \qquad\qquad\textbf{break};
\STATE    \qquad\textbf{Else}
\STATE    \qquad \qquad calculate $\lambda_n^\d$ by (\ref{lambda2});
\STATE    \qquad \qquad $i_n^\d = i_{n-1}^\d+j_{\max}$;
\STATE    \qquad \textbf{End If}
\STATE \textbf{End For}
\STATE \textbf{Output:} $\lambda_n^\d$, $i_n^\d$
\end{algorithmic}
\caption{Discrete backtracking search (DBTS-$\Theta$) algorithm for $\lambda_n^\d$, $n\geq 1$.}
\label{DBTSAlgorithm}
\end{algorithm}

We need to show that the combination parameter $\lambda_n^\d$ chosen by Algorithm \ref{DBTSAlgorithm} satisfies
(\ref{add ass1}), \eqref{condition} and \eqref{add ass3}. From Algorithm \ref{DBTSAlgorithm}
it is easily seen that $0\le \lambda_n^\d \le \beta(i_n^\d)$. Therefore (\ref{add ass1}) holds automatically.
When $\|F(z_n^\d)-y^\d\| \le \tau \d$, Algorithm \ref{DBTSAlgorithm} gives $\lambda_n^\d =0$ which ensures (\ref{condition})
hold. When $\|F(z_n^\d)-y^\d\| >\tau \d$, Algorithm \ref{DBTSAlgorithm}
either finds $\lambda_n^\d$ of the form $\beta_n(i_n^\d)$ to satisfy (\ref{condition}) or gives $\lambda_n^\d$
by (\ref{lambda2}) which again satisfies (\ref{condition}). Thus Algorithm \ref{DBTSAlgorithm} always produces
a $\lambda_n^\d$ satisfying (\ref{condition}).

For the noise-free case, Algorithm \ref{DBTSAlgorithm} either produce $\lambda_n=\beta_n(i_n)$ or $\lambda_n = 0$. Thus
$0\le \lambda_n \le \beta_n(i_n)$. By the definition of $i_n$ one can see that $i_{n}\ge i_{n-1}+1$ and thus $i_n\geq n$.
Therefore, by using $\beta_n(i) \le q(i)/\|\xi_n-\xi_{n-1}\|$ and the monotonicity of $q$, we have
\begin{align*}
\sum_{n=0}^\infty\lambda_n\|\xi_n-\xi_{n-1}\|
&\le\sum_{n=0}^\infty\beta_n(i_n)\|\xi_n-\xi_{n-1}\| \le \sum_{n=0}^\infty q(i_n)
\le  \sum_{n=0}^\infty q(n) <\infty.
\end{align*}
Therefore \eqref{add ass3} is satisfied.

We can not use Theorem \ref{T6.4} to conclude the regularization property of the two-point method (\ref{TPGM}) when the
combination parameter is determined by Algorithm \ref{DBTSAlgorithm} because the produced parameter $\lambda_n^\d$ is not
necessarily continuously dependent on $\d$ as $\d\rightarrow 0$. In fact, $\lambda_n^\d$ may have many different cluster points
as $\d\rightarrow 0$. Using these different cluster points as the combination parameter in (\ref{TPGM_E}) may lead to many
different iterative sequences for noise-free case. We need to consider all these possible iterative sequences altogether.
We will use $\Gamma_{\bar \mu_0, \bar \mu_1, \nu, q}(\xi_0, x_0)$ to denote the set consisting of all the iterative sequences
$\{(\xi_n, x_n)\} \subset \X^*\times \X$ defined by (\ref{TPGM_E}), where the combination parameters $\{\lambda_n\}$
are chosen to satisfy
\begin{align}\label{11.6.0}
\frac{1}{p^*(2 c_0)^{p^*-1}} \left(\lambda_n^{p^*}+\lambda_n\right) \|\xi_n-\xi_{n-1}\|^{p^*} \le \frac{c_1}{\nu} \mu_n \|F(z_n)-y\|^s
\end{align}
and
\begin{align}\label{11.6.00}
0\le \lambda_n \le \min\left\{\frac{q(i_n)}{\|\xi_n-\xi_{n-1}\|}, \frac{p^*(2c_0)^{p^*} \rho^p}{4 \|\xi_n-\xi_{n-1}\|^{p^*}},
\frac{n}{n+\a} \right\}
\end{align}
with a sequence $\{i_n\}$ of integers satisfying $i_0=0$ and $1\le i_n-i_{n-1} \le j_{\max}$ for all $n$.

Given a sequence $\{(\xi_n, x_n)\}\in \Gamma_{\bar \mu_0, \bar \mu_1, \nu, q}(\xi_0, x_0)$, it is easy to check that the corresponding
combination parameters $\{\lambda_n\}$ satisfy (\ref{add ass1}), (\ref{condition}) and (\ref{add ass3}).
Therefore we may use Theorem \ref{th noisefree} to conclude the convergence of $\{x_n\}$. In the following we will show a uniform
convergence result.

\begin{proposition}\label{uniform}
Let all the conditions in Theorem \ref{th noisefree} hold. If $\N(L(x^\dag)) \subset \N(L(x))$ for all $x \in B_{2\rho}(x_0)\cap \D(\Theta)$,
then for any $\eps>0$ there is an integer $n(\eps)$ such that for any sequence $\{(\xi_n, x_n)\}
\in \Gamma_{\bar \mu_0, \bar \mu_1, \nu, q}(\xi_0, x_0)$ there holds $D_{\xi_n}\Theta(x^\dag, x_n) <\eps$ for all $n\ge n(\eps)$.
\end{proposition}

\begin{proof}
Assume that the result is not true. Then there is an $\varepsilon_0>0$ such that for any $\ell\ge 1$
there exist $\{(\xi_n^{(\ell)}, x_n^{(\ell)})\}\in \Gamma_{\bar \mu_0, \bar \mu_1, \nu, q}(\xi_0, x_0)$ and $n_\ell>\ell$ such that
\begin{equation}\label{eq:6.11.1}
D_{\xi_{n_\ell}^{(\ell)}} \Theta (x^\dag, x_{n_\ell}^{(\ell)}) \ge \varepsilon_0.
\end{equation}
We will construct, for each $n=0, 1, \cdots$, a strictly increasing subsequence $\{\ell_{n,k}\}$ of positive integers and
$(\hat \xi_n, \hat x_n) \in \X^* \times \X$ such that

\begin{enumerate}

\item[(i)] $\{(\hat \xi_n, \hat x_n)\} \in \Gamma_{\bar \mu_0, \bar \mu_1, \nu, q}(\xi_0, x_0)$;

\item[(ii)] for each fixed $n$ there hold $x_n^{(\ell_{n,k})}\rightarrow \hat x_n$, $\Theta(x_n^{(\ell_{n,k})}) \rightarrow
\Theta(\hat x_n)$ and $\xi_n^{(\ell_{n,k})}\rightarrow \hat \xi_n$ as $k\rightarrow \infty$.
\end{enumerate}

\noindent
Assume that the above construction is available, we will derive a contradiction. According to (i), we may use Theorem \ref{th noisefree}
to conclude that $D_{\hat \xi_n} \Theta(x^\dag, \hat x_n) \rightarrow 0$ as $n\rightarrow \infty$. Thus we can pick a large
integer $\hat n$ such that
$$
D_{\hat \xi_{\hat n}}\Theta(x^\dag, \hat x_{\hat n})< \varepsilon_0/2.
$$
Thus
\begin{align}\label{2018.3.16.1}
\varepsilon_0/2 & > \left(D_{\hat \xi_{\hat n}} \Theta(x^\dag, \hat x_{\hat n})
-D_{\xi_{\hat n}^{(\ell_{\hat n, k})}} \Theta(x^\dag, x_{\hat n}^{(\ell_{\hat n, k})})\right)
+D_{\xi_{\hat n}^{(\ell_{\hat n, k})}} \Theta(x^\dag, x_{\hat n}^{(\ell_{\hat n, k})})   \nonumber \\
& = -D_{\xi_{\hat n}^{(\ell_{\hat n, k})}} \Theta(\hat x_{\hat n}, x_{\hat n}^{(\ell_{\hat n, k})})
+ \l \hat \xi_{\hat n}-\xi_{\hat n}^{(\ell_{\hat n, k})}, \hat x_{\hat n} -x^\dag\r \nonumber \\
& \quad \, + D_{\xi_{\hat n}^{(\ell_{\hat n, k})}} \Theta(x^\dag, x_{\hat n}^{(\ell_{\hat n, k})}).
\end{align}
According to (ii), we have
$$
-D_{\xi_{\hat n}^{(\ell_{\hat n, k})}} \Theta(\hat x_{\hat n}, x_{\hat n}^{(\ell_{\hat n, k})})
+ \l \hat \xi_{\hat n}-\xi_{\hat n}^{(\ell_{\hat n, k})}, \hat x_{\hat n} -x^\dag\r \rightarrow 0 \quad
\mbox{as } k\rightarrow \infty.
$$
Hence we can find ${\hat k}$ with $\hat \ell := \ell_{\hat n, \hat k}\ge \hat n$ such that
$$
-D_{\xi_{\hat n}^{(\hat \ell)}} \Theta(\hat x_{\hat n}, x_{\hat n}^{(\hat \ell)})
+ \l \hat \xi_{\hat n}-\xi_{\hat n}^{(\hat \ell)}, \hat x_{\hat n} -x^\dag\r \ge -\eps_0/2.
$$
Consequently, it follows from (\ref{2018.3.16.1}) that
$$
D_{\xi_{\hat n}^{(\hat \ell)}} \Theta(x^\dag, x_{\hat n}^{(\hat \ell)})<\eps_0.
$$
Note that $n_{\hat \ell}>\hat \ell = \ell_{\hat n, \hat k} \ge \hat n$. Thus, by the monotonicity
of $\{D_{\xi_{n}^{(\hat \ell)}} \Theta(x^\dag, x_{n}^{(\hat \ell)})\}$ with respect to $n$, we can obtain
$$
D_{\xi_{n_{\hat \ell}}^{(\hat \ell)}} \Theta(x^\dag, x_{n_{\hat \ell}}^{(\hat \ell)})
\le D_{\xi_{\hat n}^{(\hat \ell)}} \Theta(x^\dag, x_{\hat n}^{(\hat \ell)}) < \varepsilon_0
$$
which is a contradiction to (\ref{eq:6.11.1}) with $\ell=\hat \ell$.

We turn to the construction of $\{\ell_{n,k}\}$ and $(\hat \xi_n, \hat x_n)$, for each $n=0, 1, \cdots$, such that (i) and (ii) hold.
We use a diagonal argument. For $n=0$, we take $(\hat \xi_0, \hat x_0)=(\xi_0, x_0)$ and $\ell_{0,k}=k$ for all $k$.
Since $\xi_0^{(k)}=\xi_0$ and $x_0^{(k)}=x_0$, (ii) holds automatically for $n=0$.

Next, assume that  we have constructed $\{\ell_{n,k}\}$ and $(\hat \xi_n, \hat x_n)$ for all $0\le n\le m$.
We will construct $\{\ell_{m+1, k}\}$ and $(\hat \xi_{m+1}, \hat x_{m+1})$. Recall the combination parameter $\lambda_m^{(\ell_{m,k})}$
and the integer $i_m^{(\ell_{m,k})}$ involved in the definition of $(\xi_{m+1}^{(\ell_{m,k})}, x_{m+1}^{(\ell_{m,k})})$. We have
$$
0\le \lambda_m^{(\ell_{m,k})}\le \frac{m}{m+\a} \quad \mbox{ and } \quad 0 \le i_m^{ (\ell_{m,k})} \le m j_{\max}
$$
for all $k$. Thus, we may pick a subsequence of $\{\ell_{m,k}\}$, denoted by $\{\ell_{m+1, k}\}$, such that
\begin{align}\label{11.7.1}
\lim_{k\rightarrow \infty} \lambda_m^{(\ell_{m+1,k})} =\hat \lambda_m \quad \mbox{ and } \quad
i_m^{(\ell_{m+1,k})} = \hat i_m \quad \forall k
\end{align}
for some number $0\le \hat \lambda_m \le m/(m+\a)$ and some integer $\hat i_m$. We define
$$
\hat \zeta_m = \hat \xi_m + \hat \lambda_m (\hat \xi_m - \hat \xi_{m-1}), \qquad \hat z_m = \nabla \Theta^*(\hat \zeta_m).
$$
By the induction hypothesis, (\ref{11.7.1}) and the continuity of $\nabla \Theta^*$ we have
\begin{align}\label{11.7.2}
\zeta_m^{(\ell_{m+1,k})} \rightarrow \hat \zeta_m \quad \mbox{ and } \quad
z_m^{(\ell_{m+1,k})} \rightarrow \hat z_m  \quad \mbox{ as } k \rightarrow \infty.
\end{align}
We define
$$
\hat \mu_m = \left\{ \begin{array}{lll}
\min\left\{\frac{\bar \mu_0 \|F(\hat z_m)-y\|^{p(s-1)}}{\|L(\hat z_m)^* J_s^{\Y}(F(\hat z_m)-y)\|^p}, \mu_1\right\} \|F(\hat z_m)-y\|^{p-s}
& \mbox{ if } F(\hat z_m) \ne y,\\
0 & \mbox{ if } F(\hat z_m) = y.
\end{array}\right.
$$
By using (\ref{11.7.2}), the continuity of $F$, $L$ and $J_s^{\Y}$, and the similar argument in the proof of Lemma \ref{stability},
we can show that
\begin{align}\label{11.7.3}
\mu_m^{(\ell_{m+1,k})}\rightarrow \hat \mu_m  \quad \mbox{as } k\rightarrow \infty.
\end{align}
We next define
\begin{align*}
\hat \xi_{m+1} = \hat \zeta_m - \hat \mu_m L(\hat z_m)^* J_s^{\Y} (F(\hat z_m)-y) \quad \mbox{and} \quad
\hat x_{m+1} = \nabla \Theta^*(\hat \xi_{m+1}).
\end{align*}
By using (\ref{11.7.2}), (\ref{11.7.3}), and the continuity of $F$, $L$, $J_s^{\Y}$ and $\nabla \Theta^*$,
it follows immediately that
\begin{align}\label{11.7.4}
\xi_{m+1}^{(\ell_{m+1,k})} \rightarrow \hat \xi_{m+1} \quad \mbox{ and } \quad x_{m+1}^{(\ell_{m+1,k})}\rightarrow \hat x_{m+1}
\quad \mbox{ as } k\rightarrow \infty.
\end{align}
By the lower semi-continuity of $\Theta$ we then have
\begin{align}\label{11.7.5}
\Theta(\hat x_{m+1}) \le \liminf_{k\rightarrow \infty} \Theta(x_{m+1}^{(\ell_{m+1,k})}).
\end{align}
On the other hand, by the definition of $x_{m+1}^{(\ell_{m+1,k})}$ we have
$$
x_{m+1}^{(\ell_{m+1,k})} = \arg\min_{x\in \X} \left\{ \Theta(x) - \l \xi_{m+1}^{(\ell_{m+1,k})}, x\r \right\}.
$$
Thus
$$
\Theta(x_{m+1}^{(\ell_{m+1,k})}) - \l \xi_{m+1}^{(\ell_{m+1,k})}, x_{m+1}^{(\ell_{m+1,k})}\r
\le \Theta(\hat x_{m+1}) - \l \xi_{m+1}^{(\ell_{m+1,k})}, \hat x_{m+1} \r.
$$
Therefore, by virtue of (\ref{11.7.4}), we have
\begin{align*}
\limsup_{k\rightarrow \infty} \Theta(x_{m+1}^{(\ell_{m+1,k})})
& \le \Theta(\hat x_{m+1})+\limsup_{k\rightarrow \infty} \l \xi_{m+1}^{(\ell_{m+1,k})}, x_{m+1}^{(\ell_{m+1,k})} -\hat x_{m+1}\r \\
& = \Theta(\hat x_{m+1}).
\end{align*}
This together with (\ref{11.7.5}) implies that
\begin{align}\label{11.7.6}
\lim_{k\rightarrow \infty} \Theta(x_{m+1}^{(\ell_{m+1,k})}) = \Theta(\hat x_{m+1}).
\end{align}
We thus complete the construction of $\{\ell_{m+1,k}\}$ and $(\hat \xi_{m+1}, \hat x_{m+1})$.

We still need to show that $\hat \lambda_m$ satisfies the requirements specified in the definition of
$\Gamma_{\bar \mu_0, \bar \mu_1, \nu, q}(\xi_0, x_0)$ in order to guarantee that the constructed sequence
is indeed in $\Gamma_{\bar \mu_0, \bar \mu_1, \nu, q}(\xi_0, x_0)$.
Recall the definition of the sequences in $\Gamma_{\bar\mu_0, \bar\mu_1, \nu, q}(\xi_0, x_0)$ we have
\begin{align*}
& \frac{1}{p^*(2 c_0)^{p^*-1}} \left(\left(\lambda_m^{(\ell_{m+1, k})}\right)^{p^*}+\lambda_m^{(\ell_{m+1, k})}\right)
\|\xi_m^{(\ell_{m+1, k})}-\xi_{m-1}^{(\ell_{m+1, k})}\|^{p^*} \\
& \le \frac{c_1}{\nu} \mu_m^{(\ell_{m+1, k})} \|F(z_m^{(\ell_{m+1, k})})-y\|^s
\end{align*}
and
$$
0\le \lambda_m^{(\ell_{m+1, k})} \le \min\left\{\frac{q(i_m^{(\ell_{m+1, k})})}{\|\xi_m^{(\ell_{m+1, k})}-\xi_{m-1}^{(\ell_{m+1, k})}\|},
\frac{p^*(2c_0)^{p^*} \rho^p}{4 \|\xi_m^{(\ell_{m+1, k})}-\xi_{m-1}^{(\ell_{m+1, k})}\|^{p^*}}, \frac{m}{m+\a}\right\}
$$
with $1\le i_m^{(\ell_{m+1, k})}-i_{m-1}^{(\ell_{m+1, k})} \le j_{\max}$ for all $k$. By using (\ref{11.7.1}), (\ref{11.7.2})
and (\ref{11.7.3}), we may take $k\rightarrow \infty$ in the above two inequalities to conclude that
\begin{align*}
\frac{1}{p^*(2 c_0)^{p^*-1}} \left(\hat \lambda_m^{p^*}+ \hat \lambda_m\right)
\|\hat \xi_m-\hat \xi_{m-1}\|^{p^*}\
\le \frac{c_1}{\nu} \hat \mu_m \|F(\hat z_m)-y\|^s
\end{align*}
and
$$
0\le \hat \lambda_m \le \min\left\{\frac{q(\hat i_m)}{\|\hat \xi_m-\hat \xi_{m-1}\|},
\frac{p^*(2c_0)^{p^*} \rho^p}{4 \|\hat \xi_m-\hat \xi_{m-1}\|^{p^*}}, \frac{m}{m+\a}\right\}
$$
Furthermore, we can assume that $\{\ell_{m,k}\}$ was taken so that $i_{m-1}^{(\ell_{m,k})} = \hat i_{m-1}$ for all $k$ for some
integer $\hat i_{m-1}$. Recall that $\{\ell_{m+1, k}\}$ is a subsequence of $\{\ell_{m,k}\}$, we must have
$1\le \hat i_m - \hat i_{m-1} =  i_m^{(\ell_{m+1, k})}-i_{m-1}^{(\ell_{m+1, k})} \le j_{\max}$. Note that $i_0^{(\ell)}=0$
for all $\ell$, we also have $\hat i_0 =0$. The proof is therefore complete.
\end{proof}

\begin{lemma}\label{stability2}
Let $\X$ be reflexive, let $\Y$ be uniformly smooth, and let Assumption \ref{A1} and Assumption \ref{A2} hold.
Let $\tau>1$ and $\bar \mu_0>0$ be chosen to satisfy (\ref{c1}). Let $\{y^{\d_l}\}$ be a sequence of noisy data
satisfying $\|y^{\d_l}-y\| \le \d_l$ with $\d_l\rightarrow 0$ as $l\rightarrow \infty$. Assume that the combination parameters
$\{\lambda_n^{\d_l}\}$ are chosen by Algorithm \ref{DBTSAlgorithm} with $i_0^{\d_l} =0$. Then, for any $n\ge 0$, by taking a
subsequence of $\{y^{\d_l}\}$ if necessary, there is a sequence $\{(\xi_k, x_k)\} \in \Gamma_{\bar \mu_0, \bar \mu_1, \nu, q}(\xi_0, x_0)$
such that
\begin{equation*}
x_k^{\d_l} \rightarrow x_k, \quad \xi_k^{\d_l} \rightarrow \xi_k
\quad \mbox{and} \quad \Theta(x_k^{\d_l})\rightarrow \Theta(x_k) \quad \textrm{as } l\rightarrow \infty
\end{equation*}
for all $0\le k\le n$.
\end{lemma}

\begin{proof}
We will use an induction argument on $n$. When $n=0$, nothing needs to be proved since $x_0^{\d_l}=x_0$ and
$\xi_0^{\d_l}=\xi_0$. Assume next that the result is true for $n = m$, i.e. there is a sequence
$\{(\xi_k, x_k)\}\in \Gamma_{\bar \mu_0, \bar \mu_1, \nu, q}(\xi_0, x_0)$ such that
$$
x_k^{\d_l} \rightarrow x_k, \quad \xi_k^{\d_l} \rightarrow \xi_k \quad \mbox{and} \quad
\Theta(x_k^{\d_l}) \rightarrow \Theta(x_k) \quad \mbox{as  } l \rightarrow \infty
$$
for all $0\le k \le m$. In order to show that the result is also true for $n = m+1$, we will obtain a sequence from
$\Gamma_{\bar \mu_0, \bar \mu_1, \nu, q}(\xi_0, x_0)$ by retaining the first $m+1$ terms in $\{(\xi_k, x_k)\}$
and modifying the remaining terms. It suffices to redefine $\xi_{m+1}$ and $x_{m+1}$ because we can then apply
the method (\ref{TPGM_E}) with $\lambda_n =0$ for $n\ge m+1$ to produce the remaining terms.

Note that the combination parameter $\lambda_m^{\d_l}$ determined by Algorithm \ref{DBTSAlgorithm} satisfies
\begin{align}\label{11.8.0}
\frac{1}{p^*(2c_0)^{p^*-1}}\left(\left(\lambda_m^{\d_l}\right)^{p^*} + \lambda_m^{\d_l}\right) \|\xi_m^{\d_l}-\xi_{m-1}^{\d_l}\|^{p^*}
\le \frac{c_1}{\nu} \mu_m^{\d_l} \|F(z_m^{\d_l})-y^{\d_l}\|^s
\end{align}
and
\begin{align}\label{11.8.00}
\begin{split}
& 0\le \lambda_m^{\d_l} \le \min\left\{\frac{q(i_m^{\d_l})}{\|\xi_m^{\d_l}-\xi_{m-1}^{\d_l}\|},
\frac{p^*(2c_0)^{p^*} \rho^p}{4\|\xi_m^{\d_l}-\xi_{m-1}^{\d_l}\|^{p^*}}, \frac{m}{m+\a}\right\} \,  \mbox{ or }\\
&\lambda_m^{\d_l} = \min\left\{ \frac{(2 c_0)^{p^*-1} p^* M\tau^p \d_l^p}{2 \|\xi_m^{\d_l} -\xi_{m-1}^{\d_l}\|^{p^*}}, \frac{m}{m+\a}\right\}
\end{split}
\end{align}
with $1 \le i_m^{\d_l} - i_{m-1}^{\d_l} \le j_{\max}$. By taking a subsequence of $\{y^{\d_l}\}$ if necessary,  we have
\begin{align}\label{11.8.1}
\lim_{l\rightarrow \infty} \lambda_m^{\d_l} = \lambda_m \quad \mbox{ and } \quad
i_m^{\d_l} = i_m \quad \mbox{for all } l
\end{align}
for some number $0\le \lambda_m \le m/(m+a)$ and some integer $i_m$. We now define $\zeta_m$, $z_m$, $\mu_m$, $\xi_{m+1}$ and
$x_{m+1}$ by (\ref{TPGM_E}) with $n=m$. By using (\ref{11.8.1}) and the similar argument in the proof of Lemma \ref{stability},
we have
\begin{align*}
\zeta_m^{\d_l} \rightarrow \zeta_m, \quad z_m^{\d_l} \rightarrow z_m, \quad \mu_m^{\d_l} \rightarrow \mu_m, \quad
\xi_{m+1}^{\d_l} \rightarrow \xi_{m+1} \quad \mbox{and} \quad x_{m+1}^{\d_l} \rightarrow x_{m+1}
\end{align*}
as $l\rightarrow \infty$. Furthermore, we can use the similar argument for proving (\ref{11.7.6}) to obtain
$\Theta(x_{m+1}^{\d_l}) \rightarrow \Theta(x_{m+1})$ as $l\rightarrow \infty$. Now by taking $l\rightarrow \infty$
in (\ref{11.8.0}) and (\ref{11.8.00}) and using the induction hypothesis, we can conclude that $\lambda_m$
and $i_m$ satisfy the requirements specified in the definition of $\Gamma_{\bar\mu_0, \bar \mu_1, \nu q}(\xi_0, x_0)$.
\end{proof}

We are now ready to show the regularization property of the method (\ref{TPGM}) when the combination parameter $\lambda_n^\d$
is chosen by Algorithm \ref{DBTSAlgorithm}.

\begin{theorem}\label{last}
Let $\X$ be reflexive, let $\Y$ be uniformly smooth, and let Assumption \ref{A1} and Assumption \ref{A2} hold.
Let $\tau>1$ and $\bar \mu_0>0$ be chosen to satisfy (\ref{c1}). Assume that the combination parameters $\{\lambda_n^{\d_l}\}$
are chosen by Algorithm \ref{DBTSAlgorithm} with $i_0^{\d_l} =0$. Let $n_\d$ be the integer determined by the
discrepancy principle (\ref{dp}). If $\N(L(x^\dag)) \subset \N(L(x))$ for all $x\in B_{3\rho}(x_0)$, then
$$
\lim_{\d\rightarrow 0} \|x_{n_\d}^\d - x^\dag\| =0 \quad \mbox{ and } \quad
\lim_{\d\rightarrow 0} D_{\xi_{n_\d}^\d}\Theta(x^\dag, x_{n_\d}^\d) =0.
$$
\end{theorem}

\begin{proof}
We will show that
\begin{align}\label{2018.3.8.2}
D_{\xi_{n_\d}^\d}\Theta(x^\dag, x_{n_\d}^\d) \rightarrow 0 \quad \mbox{ as } \d \rightarrow 0
\end{align}
by a contradiction argument. If it is not true, then there is a number $\eps>0$ and a subsequence $\{y^{\d_l}\}$ of $\{y^\d\}$
satisfying $\|y^{\d_l}-y\| \le \d_l$ with $\d_l\rightarrow 0$ as $l\rightarrow \infty$ such that
\begin{align}\label{2018.3.8.1}
D_{\xi_{n_{\d_l}}^{\d_l}} \Theta(x^\dag, x_{n_{\d_l}}^{\d_l}) \ge \eps  \quad \mbox{ for all } l.
\end{align}
We need to consider two cases.

{\it Case 1.} $\{n_{\d_l}\}$ has a finite cluster point $\hat n$. For this case, by using Lemma \ref{stability2} and by
taking a subsequence if necessary, we have $n_{\d_l}=\hat n$ for all $l$ and there is a sequence
$\{(\xi_k, x_k)\}\in \Gamma_{\bar \mu_0, \bar \mu_1, \nu, q}(\xi_0, x_0)$ such that
$$
x_{\hat n}^{\d_l} \rightarrow x_{\hat n}, \quad \xi_{\hat n}^{\d_l} \rightarrow \xi_{\hat n} \quad \mbox{ and } \quad
\Theta(x_{\hat n}^{\d_l}) \rightarrow \Theta(x_{\hat n}) \quad \mbox{ as } l \rightarrow \infty.
$$
Using the similar argument in the proof of Theorem \ref{T6.4} we can obtain
$x_{\hat n} = x^\dag$. Therefore
$$
x_{n_{\d_l}}^{\d_l} \rightarrow x^\dag,  \quad \xi_{n_{\d_l}}^{\d_l} \rightarrow \xi_{\hat n} \quad \mbox{ and } \quad
\Theta(x_{n_{\d_l}}^{\d_l}) \rightarrow \Theta(x^\dag) \quad \mbox{ as } l\rightarrow \infty.
$$
Consequently $D_{\xi_{n_{\d_l}}^{\d_l}}\Theta(x^\dag, x_{n_{\d_l}}^{\d_l}) \rightarrow 0 $ as $l\rightarrow \infty$ which contradicts (\ref{2018.3.8.1}).

{\it Case 2.} $\lim_{l\rightarrow \infty} n_{\d_l}=\infty$. According to Proposition \ref{uniform}, there is an integer
$n(\eps)$ such that
\begin{equation}\label{6.26.51}
D_{\xi_{n(\varepsilon)}} \Theta(x^\dag, x_{n(\eps)}) <\eps \quad
\mbox{ for any } \{(\xi_n, x_n)\}\in \Gamma_{\bar \mu_0, \bar\mu_1, \nu, q}(\xi_0, x_0).
\end{equation}
For this $n(\eps)$, by using Lemma \ref{stability2} and by taking a subsequence if necessary, we can find
$\{(\xi_n, x_n)\}\in \Gamma_{\bar \mu_0, \bar \mu_1, \nu, q}(\xi_0, x_0)$ such that
\begin{equation}\label{6.26.41}
x_n^{\d_l} \rightarrow x_n, \quad \xi_n^{\d_l} \rightarrow \xi_n \quad \mbox{and} \quad \Theta(x_n^{\d_l}) \rightarrow \Theta(x_n)
\quad \mbox{as } l \rightarrow \infty
\end{equation}
for $0\le n\le n(\eps)$. Since $\lim_{l\rightarrow \infty} n_{\d_l}=\infty$, we have $n_{\d_l}>n(\eps)$ for large $l$.
Thus, it follows from (\ref{10.11.2}) in Proposition \ref{monotonicity} that
$$
D_{\xi_{n_{\d_l}}^{\d_l}} \Theta(x^\dag, x_{n_{\d_l}}^{\d_l}) \le D_{\xi_{n(\eps)}^{\d_l}}\Theta(x^\dag, x_{n(\eps)}^{\d_l}).
$$
Letting $l\rightarrow \infty$ and using (\ref{6.26.41}) and (\ref{6.26.51}) we can obtain
\begin{align*}
\limsup_{l\rightarrow \infty} D_{\xi_{n_{\d_l}}^{\d_l}}\Theta(x^\dag, x_{n_{\d_l}}^{\d_l})
\le \limsup_{l\rightarrow \infty} D_{\xi_{n(\eps)}^{\d_l}} \Theta(x^\dag, x_{n(\eps)}^{\d_l})
= D_{\xi_{n(\eps)}}\Theta(x^\dag, x_{n(\eps)}) <\eps
\end{align*}
which is a contradiction to (\ref{2018.3.8.1}).

Combining the above two cases, we therefore obtain (\ref{2018.3.8.2}). By virtue of the $p$-convexity of $\Theta$,
we then have $\|x_{n_\d}^{\d} -x^\dag\| \rightarrow 0$ as $\d\rightarrow 0$.
\end{proof}

\begin{remark}
{\rm A two-point gradient method in Hilbert spaces was considered in \cite{hr17} in which the combination parameter
is chosen by a discrete backtracking search (DBTS) algorithm. The regularization property was proved under the condition
that the noise-free counterpart of the method never terminates at a solution of (\ref{sys}) in finite many steps if the
combination parameter is chosen by the DBTS algorithm. This technical condition seems difficult to be verified
because the exact data $y$ is unavailable. In Theorem \ref{last} we removed this condition by using a uniform
convergence result established in Proposition \ref{uniform}.}
\end{remark}

\section{\bf Numerical simulations} \label{sec:numerics}
\setcounter{equation}{0}

In this section we will present numerical simulations on our TPG-DBTS method, i.e. the two point gradient method (\ref{TPGM})
with the combination parameter $\lambda_n^\d$ chosen by the DBTS algorithm (Algorithm \ref{DBTSAlgorithm}). In order to
illustrate the performance of TPG-DBTS algorithm, we will compare the computational results with the ones obtained by the
Landweber iteration  (\ref{lenear1}) and the Nesterov acceleration of Landweber iteration, i.e. the method (\ref{TPGM}) with $\lambda_n^\d
= n/(n+\a)$ for some $\a\ge 3$. In order to be fair, the step sizes $\mu_n^\d$ involved in all these methods are
computed by (\ref{step}) and all the iterations are terminated by the discrepancy principle with $\tau = 1.05$.

A key ingredient for the numerical implementation is the determination of $x= \nabla \Theta^*(\xi)$ for any given $\xi\in \X^*$
which is equivalent to solving the minimization problem
\begin{equation}\label{eq4.16.1}
x=\arg \min_{z\in \X} \left\{\Theta(z) -\l \xi, z\r\right\}.
\end{equation}
For some choices of $\Theta$, this minimization problem can be easily solved numerically. For instance, when $\X=L^2(\Omega)$
and the sought solution is piecewise constant, we may choose
\begin{align}\label{TV}
\Theta(x) = \frac{1}{2\beta} \|x\|_2^2 + |x|_{TV}
 \end{align}
with a constant $\beta>0$, where $|x|_{TV}$ denotes the total variation of $x$. Then the minimization problem (\ref{eq4.16.1})
becomes the total variation denoising problem
\begin{align}\label{TVdenoising}
x = {\rm arg}\min_{z\in L^2(\Omega)}\left\{\frac{1}{2\beta}\|z-\beta\xi\|_2^2 + |z|_{TV}\right\}
\end{align}
which is nonsmooth and convex. Note that for this $\Theta$, Assumption \ref{A1} holds with $p=2$ and $c_0 = \frac{1}{2\beta}$.
Many efficient algorithms have been developed for solving (\ref{TVdenoising}), including the fast iterative shrinkage-thresholding
algorithm \cite{bt09,bt09(2)}, the alternating direction method of multipliers \cite{BPCPE2011}, and the primal dual hybrid gradient (PDHG)
method \cite{ZC2008}.

In the following numerical simulations we will only consider the situation that the sought solution is piecewise constant.
We will use the PDHG method to solve (\ref{TVdenoising}) iteratively. Our simulations are performed via MATLAB R2012a
on a Lenovo laptop with Intel(R) Core(TM) i5 CPU 2.30GHz and 6GB memory.

\subsection{Computed tomography}

Computed tomography (CT) consists in determining the density of cross sections of a human body by measuring the attenuation of X-rays as
they propagate through the biological tissues. Mathematically, it requires to determine a function supported on a bounded domain from
its line integrals \cite{N2001}. In order to apply our method to solve the CT problems, we need a discrete model. In our numerical experiment,
we assume that the image is supported on a rectangular domain in $\mathbb{R}^2$ which is divided into $I\times J$ pixels so that the discrete
image has size $I\times J$ and can be represented by a vector $x\in \mathbb{R}^{N}$ with $N=I\times J$. We further assume that
there are $n_{\theta}$ projection directions and in each direction there are $p$ X-rays emitted. We want to reconstruct the image by using
the measurement data of attenuation along the rays which can be represented by a vector $b\in \mathbb{R}^{M}$ with $M=n_{\theta}\times p$.
According to a standard discretization of the Radon  transform (\cite{Hansen2012}), we arrive at a linear algebraic system
$$
Fx=b,
$$
where $F$ is a sparse matrix of size $M\times N$ whose form depends on the scanner geometry.

\begin{table}[h]
\caption{ Computed tomography ($\beta=1$, $\tau=1.05$).}\label{t0}
\begin{center}
\begin{tabular}{lllll}
      \hline
$\d_{rel}$ & method      &  $n_\d$  & CPU time (s) &  $\|x_n^\d-x^\dag\|_2/\|x^\dag\|_2$  \\
\hline
0.05      &  Landweber   &    109   &   22.33      & 0.19431     \\
~         &  Nesterov    &    34    &   7.45       & 0.17573     \\
~         &  TPG-DBTS    &    34    &   11.99      & 0.17573     \\
0.01      &  Landweber   &    489   &   106.97     & 0.06410      \\
~         &  Nesterov    &    79    &   16.25      & 0.05923      \\
~         &  TPG-DBTS    &    79    &   29.51      & 0.05923      \\
0.005     &  Landweber   &    879   &   197.65     & 0.03578      \\
~         &  Nesterov    &    109   &   22.25      & 0.03165      \\
~         &  TPG-DBTS    &    109   &   43.02      & 0.03165      \\
0.001     &  Landweber   &    3299  &   775.10     & 0.00694      \\
~         &  Nesterov    &    247   &   51.05      & 0.00521      \\
~         &  TPG-DBTS    &    247   &   105.29     & 0.00521      \\
0.0005    &  Landweber   &    5703  &   1328.31    & 0.00326     \\
~         &  Nesterov    &    351   &   82.93      & 0.00199     \\
~         &  TPG-DBTS    &    351   &   153.79     & 0.00199     \\
\hline
\end{tabular}\\[5mm]
\end{center}
\end{table}

In the numerical simulations we consider only test problems that model the standard 2D parallel-beam tomography. The true image is
taken to be the modified Shepp-Logan phantom of size $256\times 256$ generated by MATLAB. This phantom is widely used in evaluating
tomographic reconstruction algorithms. We use the full angle model with 45 projection angles evenly distributed between 1 and 180 degrees,
with 367 lines per projection. The function \texttt{paralleltomo} in MATLAB package AIR TOOLS \cite{Hansen2012} is used to generate the
sparse matrix $F$, which has the size $M=16515$ and $N=66536$. Let $x^{\dag }$ denote the vector formed by stacking all the columns of
the true image and let $b = F x^\dag$ be the true data. We add Gaussian noise on $b$ to generate a noisy data $b^\d$ with relative noise level
$\delta_{rel}=\|b^\d-b\|_2/\|b\|_2$ so that the noise level is $\d = \d_{rel} \|b\|_2$. We will use $b^\d$ to reconstruct $x^\dag$.
In order to capture the feature of the sought image, we take $\Theta$ to be the form (\ref{TV}) with $\beta =1$.
In our numerical simulations we will use $\xi_0 =0$ as an initial guess. According to (\ref{c1}) we need $\bar \mu_0 <2 (1-1/\tau)/\beta$.
Therefore we take the parameters $\bar \mu_0$ and $\bar \mu_1$ in the definition of $\mu_n^\d$ to be $\bar \mu_0 = 1.8(1-1/\tau)/\beta$
and $\bar \mu_1 = 20000$. For implementing TPG-DBTS method with $\lambda_n^\d$ chosen by Algorithm
\ref{DBTSAlgorithm}, we take $j_{\max} = 1$, $\a=5$, $\gamma_0 = 0.1$ in (\ref{lambda2}), $\gamma_1 = 0.4$; we also choose the function
$q: {\mathbb N}\to {\mathbb N}$ by $q(m) = m^{-1.1}$. For implementing the Nesterov acceleration of Landweber iteration, we take
$\lambda_n^\d = n/(n+\a)$ with $\a = 5$. During the computation, the total variation denoising problem (\ref{TVdenoising}) involved in
each iteration step is solved approximately by the PDHG method after 100 iterations.

\begin{figure}[htp]
  \begin{center}
  \includegraphics[width = 1\textwidth]{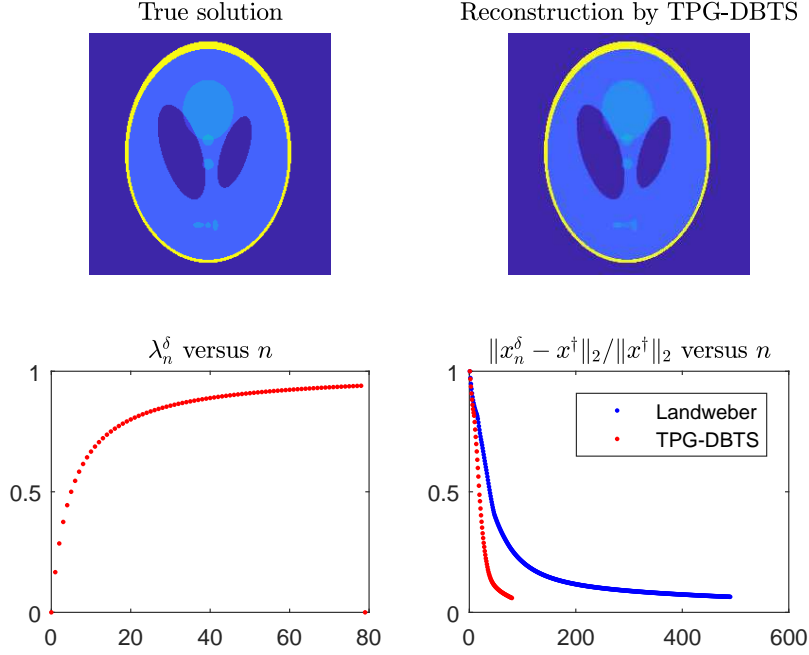}
  \end{center}
  \caption{The computed tomography using noisy data with relative noise level $\d_{rel} = 0.01$.} \label{fig0}
\end{figure}

The computational results by TPG-DBTS, Landweber, and Nesterov acceleration of Landweber are reported in Table \ref{t0}, including
the number of iterations $n_\d$, the CPU running time and the relative errors $\|x_{n_\d}^\d-x^\dag\|_2/\|x^\dag\|_2$, using noisy data with
various relative noise level $\d_{rel}>0$. Table \ref{t0} shows that both TPG-DBTS and Nesterov acceleration, terminated by the discrepancy principle,
lead to a considerable decrease in the number of iterations and the amount of computational time, which demonstrates that these two methods
have the striking acceleration effect. Moreover, both TPG-DBTS and Nesterov acceleration produce more accurate results than Landweber iteration.
With the above setup, our computation shows that TPG-DBTS produces the combination parameter $\lambda_n^\d$ which is exactly
same as the combination parameter $n/(n+\a)$ in Nesterov acceleration in each iteration step. Therefore, TPG-DBTS and Nesterov acceleration
require the same number of iterations and produce the same reconstruction result. Because TPG-DBTS spends more time on determining $\lambda_n^\d$,
the Nesterov acceleration requires less amount of computational time than TPG-DBTS. However, unlike TPG-DBTS, there exists no convergence result
concerning Nesterov acceleration for ill-posed inverse problems.

In order to visualize the reconstruction accuracy of the TPG-DBTS method, we plot in Figure \ref{fig0} the true image, the reconstruction
result by TPG-DBTS using noisy data with relative noise level $\d_{rel} = 0.01$, the curve of $\lambda_n^\d$ versus $n$, and the relative error
$\|x_n^\d-x^\dag\|_2/\|x^\dag\|_2$ versus $n$ for TPG-DBTS and Landweber iteration.

\subsection{Elliptic parameter identification}

We consider the identification of the parameter $c$ in the elliptic boundary value problem
\begin{align}\label{PDE}
\left\{
  \begin{array}{ll}
    -\triangle u + cu = f & \hbox{in $\Omega$}, \\
    u=g & \hbox{on $\partial \Omega$}
  \end{array}
\right.
\end{align}
from an $L^2(\Omega)$-measurement of the state $u$, where $\Omega\subset\mathbb{R}^d$ with $d\leq 3$ is
a bounded domain with Lipschitz boundary $\p \Omega$, $f\in H^{-1}(\Omega)$ and $g\in H^{1/2}(\Omega)$. We assume that
the sought parameter $c^{\dag}$ is in $L^2(\Omega)$. This problem reduces to solving $F(c) =u$ if we define the
nonlinear operator $F: L^2(\Omega)\rightarrow  L^2(\Omega)$ by
\begin{equation}\label{nonlinear}
F(c): = u(c)\,,
\end{equation}
where $u(c)\in H^1(\Omega)\subset  L^2(\Omega)$ is the unique solution of (\ref{PDE}). This operator $F$
is well defined on
\begin{equation*}
\D: = \left\{ c\in  L^2(\Omega) : \|c-\hat{c}\|_{ L^2(\Omega)}\leq \eps_0 \mbox{ for some }
\hat{c}\geq0,\ \textrm{a.e.}\right\}
\end{equation*}
for some positive constant $\eps_0>0$. It is well-known (\cite{ehn96}) that the operator $F$ is weakly closed
and Fr{\'{e}}chet differentiable with
$$
F'(c) h = v    \quad \mbox{ and } \quad F'(c)^* \sigma = - u(c) w
$$
for $c\in \D$ and $h, \sigma \in L^2(\Omega)$, where $v, w\in H^1(\Omega)$ are the unique
solutions of the problems
\begin{align*}
\left\{  \begin{array}{ll}
    -\triangle v + cv = -h u(c) \quad \mbox{ in } \Omega, \\[0.6ex]
    v=0 \quad \mbox{ on } \p \Omega
\end{array}\right. \quad \mbox{ and } \quad
\left\{  \begin{array}{ll}
    -\triangle w + cw = \sigma  \quad \mbox{ in } \Omega, \\[0.6ex]
    w=0 \quad \mbox{ on } \p \Omega
\end{array} \right.
\end{align*}
respectively. Moreover, $F$ satisfies Assumption \ref{A2} (c).

\begin{table}[h]
\caption{ 2-dimensional elliptic parameter estimation ($\beta=10$, $\tau=1.05$).}\label{t1}
\begin{center}
\begin{tabular}{lllll}
      \hline
$\d$      &   method     &  $n_\d$  & CPU time (s)& $\|c_n^\d-c^\dag\|_{L^2}$   \\
\hline
0.005     &  Landweber   &    34    &   9.41      & 0.26884      \\
~         &  Nesterov    &    20    &   5.74      & 0.24815      \\
~         &  TPG-DBTS    &    21    &   6.22      & 0.24815      \\
0.001     &  Landweber   &    184   &   54.49     & 0.12831      \\
~         &  Nesterov    &    72    &   20.58     & 0.11466      \\
~         &  TPG-DBTS    &    72    &   24.33     & 0.11466      \\
0.0005    &  Landweber   &    263   &   74.15     & 0.11096      \\
~         &  Nesterov    &    111   &   29.96     & 0.09485      \\
~         &  TPG-DBTS    &    111   &   41.19     & 0.09485      \\
0.0001    &  Landweber   &   1042   &   324.49    & 0.08107     \\
~         &  Nesterov    &    222   &   65.26     & 0.07688     \\
~         &  TPG-DBTS    &    222   &   103.64    & 0.07688     \\
0.00005   &  Landweber   &    2346  &   785.07    & 0.07221     \\
~         &  Nesterov    &    452   &   147.03    & 0.06609     \\
~         &  TPG-DBTS    &    434   &   242.24    & 0.06342     \\
\hline
\end{tabular}\\[5mm]
\end{center}
\end{table}

In our numerical simulation, we consider the two-dimensional problem with $\Omega = [0, 1]\times [0,1]$ and
the sought parameter is assumed to be
\begin{align*}
c^\dag(x,y) = \left\{
              \begin{array}{ll}
                1, & \hbox{ if } (x-0.65)^2 + (y-0.36)^2 \le 0.18^2, \\
                0.5, & \hbox{ if } (x-0.35)^2 + 4(y-0.75)^2 \le 0.2^2, \\
                0, & \hbox{elsewhere.}
              \end{array}
            \right.
\end{align*}
Assuming $u(c^\dag) = x+y$, we add random Gaussian noise to produce noisy data $u^\d$ satisfying $\|u^\d-u(c^\dag)\|_{L^2(\Omega)}\le \d$
with various noise level $\d>0$. We will use $u^\d$ to reconstruct $c^\dag$. In order to capture the feature of the sought parameter,
we take $\Theta$ to be the form (\ref{TV}) with $\beta = 10$. We will use the initial guess $\xi_{0}=0$ to carry out the iterations.
The parameters $\bar \mu_0$ and $\bar \mu_1$ in the definition of $\mu_n^\d$ are taken to be
$\bar \mu_0 = (1-1/\tau)/\beta$ and $\bar \mu_1 = 20000$. For implementing TPG-DBTS method with $\lambda_n^\d$ chosen by Algorithm
\ref{DBTSAlgorithm}, we take $j_{\max} = 1$, $\a=5$, $\gamma_0 = 0.1$ in (\ref{lambda2}), $\gamma_1 = 0.3$; we also choose the function
$q: {\mathbb N}\to {\mathbb N}$ by $q(m) = m^{-1.2}$. For implementing the Nesterov acceleration of Landweber iteration, we take
$\lambda_n^\d = n/(n+\a)$ with $\a = 5$. In order to carry out the computation, we divide $\Omega$ into $128\times 128$ small squares
of equal size and solve all partial differential equations involved approximately by a multigrid method (\cite{H2016}) via finite difference
discretization. The total variation denoising problem (\ref{TVdenoising}) involved in each iteration step is solved by the PDHG method
after 200 iterations.

\begin{figure}[htp]
  \begin{center}
  \includegraphics[width = 1\textwidth]{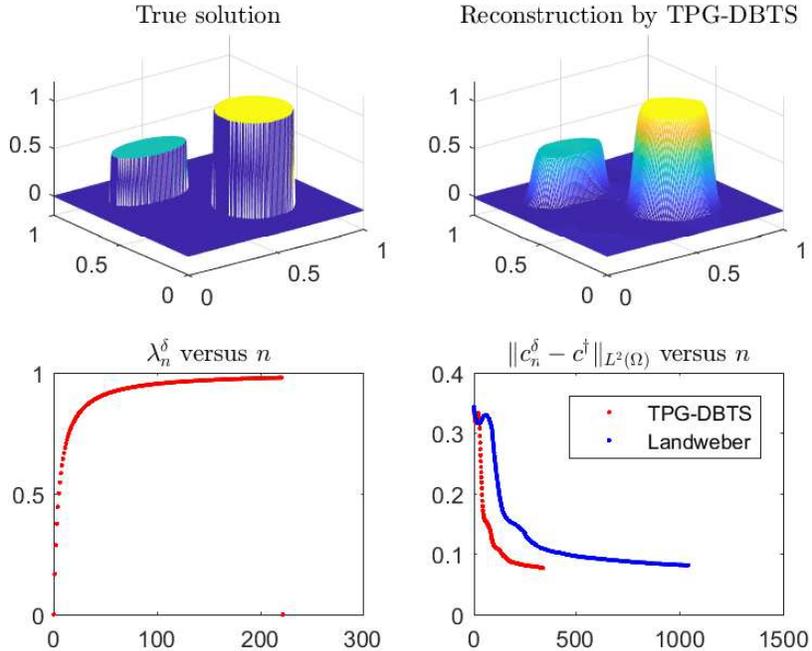}
  \end{center}
  \caption{The 2-dimensional elliptic parameter identification using noisy data with noise level $\delta = 0.0001$.} \label{fig1}
\end{figure}

In Table \ref{t1} we report the computational results by TPG-DBTS, Landweber, and Nesterov acceleration of Landweber, including
the number of iterations $n_\d$, the CPU running time and the absolute errors $\|c_{n_\d}^\d-c^\dag\|_{L^2(\Omega)}$, for various
noise level $\d>0$. Table \ref{t1} shows that both TPG-DBTS and Nesterov acceleration, terminated by the discrepancy principle,
reduce the number of iterations and the amount of computational time significantly, and produce more accurate results
than Landweber iteration. This demonstrates that these two methods have a remarkable acceleration effect.
In order to visualize the reconstruction accuracy of the TPG-DBTS method, we plot in Figure \ref{fig1} the true solution, the reconstruction
result by TPG-DBTS with noise level $\d = 0.0001$, the curve of $\lambda_n^\d$ versus $n$, and the error $\|c_n^\d-c^\dag\|_{L^2(\Omega)}$ versus $n$
for TPG-DBTS and Landweber iteration.

\subsection{Robin coefficient reconstruction}

We consider the heat conduction process in a homogeneous solid rod located on the interval $[0,\pi]$. If the endpoints
of the rod contacts with liquid media, then the convective heat transfer occurs. The temperature field of $u(x,t)$ during
a time interval $[0,T]$ with a fixed time of interest $T>0$ can be modeled by
\begin{align}\label{robin}
\left\{
  \begin{array}{ll}
    u_t-a^2 u_{xx} = 0, & \hbox{$x\in(0,\pi),\ t\in(0,T)$;} \\
    u_x(0,t) = f(t),\quad u_x(\pi,t) + \sigma(t)u(\pi,t) = \varphi(t), & \hbox{$t\in[0,T]$;}\\
    u(x,0) = u_0(x), & \hbox{$x\in [0,\pi]$.}
  \end{array}
\right.
\end{align}
The function $\sigma(t)\ge 0$ represents the corrosion damage, which is interpreted as a Robin coefficient of energy exchange.
We will assume $f$, $\varphi$ and $u_0$ are all continuous. Notice that if $\sigma(t)$ is given, the problem \eqref{robin} is a
well-posed direct problem. The inverse problem of identifying the Robin coefficient $\sigma(t)$ requires additional data to be
specified. We consider the reconstruction of $\sigma(t)$ from the temperature information measured at the boundary
\begin{align*}
u(0,t) = g(t),\quad t\in[0,T]\,.
\end{align*}
Define
\begin{align}
\D: = \{\sigma \in L^2[0,T], 0<\sigma_-\le \sigma\le \sigma_+,\ a.e.\ \textrm{in }[0,T]\},
\end{align}
and define the nonlinear operator $F:\sigma\in\D \to u[\sigma](0,t)\in L^2[0,T]$, where $u[\sigma]$ denotes the unique solution of \eqref{robin}.
Then the above Robin coefficient inversion problem reduces to solving the equation $F(\sigma) = g$. We refer to \cite{wl17} for the
well-posedness of $F$ and the uniqueness of the inverse problem in the $L^2$ sense.
By the standard theory of parabolic equation, one can show that $F$ is Fr\'{e}chet differentiable in the sense that
\begin{align*}
\|F(\sigma+h) - F(\sigma) - F'(\sigma) h \|_{L^2(0,T)} = o(\|h\|_{L^2(0,T)})
\end{align*}
for all $\sigma,\sigma+h\in\D$, where $[F'(\sigma)h](t) = w(0,t)$ and $w$ is the unique solution of
\begin{align*}
\left\{
  \begin{array}{ll}
    w_t-a^2 w_{xx} = 0, & \hbox{$x\in(0,\pi),\ t\in(0,T)$;} \\
    w_x(0,t) = 0,\, \, \,  w_x(\pi,t) + \sigma(t) w(\pi,t) = - h(t) u[\sigma](\pi, t), & \hbox{$t\in[0,T]$;}\\
    w(x,0) = 0, & \hbox{$x\in [0,\pi]$.}
  \end{array}
\right.
\end{align*}
In addition, the adjoint of the Fr\'{e}chet derivative is given by
\begin{align*}
[F'(\sigma)^*\zeta](t) =  u[\sigma](\pi,t)v(\pi,\tau),
\end{align*}
where $v(x,t)$ solves the adjoint system
\begin{align*}
\left\{
  \begin{array}{ll}
    -v_t-a^2 v_{xx} = 0, & \hbox{$x\in(0,\pi),\ t\in(0,T)$;} \\
    v_x(0,t) = \zeta(t),\, \, \, v_x(\pi,t) + \sigma(t) v(\pi,t) =0, & \hbox{$t\in[0,T]$;}\\
    v(x,T) = 0, & \hbox{$x\in [0,\pi]$.}
  \end{array}
\right.
\end{align*}

\begin{table}[h]
  \caption{Robin coefficient reconstruction ($\beta=1, \tau=1.05$).}\label{table3}
    \begin {center}
\begin{tabular}{lllll}
      \hline
$\delta$  & Method     &  $n_\d$  &CPU time (s) &  $\|\sigma_n^\d-\sigma^\dag\|_{L^2}$             \\
\hline
0.1       & Landweber  &  242     &  1.18    & 0.120974     \\
 ~        &  Nesterov  &  76      &  0.56    & 0.119851     \\
 ~        &  TPG-DBTS  &  76      &  1.03    & 0.119851       \\
0.01      & Landweber  &  1431    &  5.33    & 0.042751     \\
 ~        &  Nesterov  &  240     &  1.14    & 0.038788     \\
 ~        &  TPG-DBTS  &  235     &  2.63    & 0.038193       \\
0.001     &  Landweber &  12033   &  42.14   & 0.007127      \\
 ~        &  Nesterov  &  706     &  2.76    & 0.002754       \\
 ~        &  TPG-DBTS  &  732     &  7.44    & 0.002569        \\
0.0001    &  Landweber &  31021   &  109.85  & 0.000702     \\
 ~        &  Nesterov  &  959     &  3.65    & 0.000236     \\
 ~        &  TPG-DBTS  &  2113    &  18.89   & 0.000102      \\
   \hline
    \end{tabular}\\[5mm]
    \end{center}
\end{table}

In our numerical simulations, we take $a=5$, $T=1$, and assume the sought Robin coefficient is
\begin{align*}
\sigma^\dag(t) = \left\{
              \begin{array}{ll}
                1.5,  & 0\le t \le 0.1563, \\
                2,    & 0.1563 <t\le 0.3125, \\
                1.2,  & 0.3125 <t\le 0.5469,\\
                2.5,  & 0.5469 <t\le 0.6250,\\
                1.8,  & 0.6250 <t\le 0.7813,\\
                1,    & 0.7813 <t\le 1.
              \end{array}
            \right.
\end{align*}
We also assume that the exact solution of the forward problem (\ref{robin}) with $\sigma= \sigma^\dag$ is
\begin{align}\label{exact}
u(x,t) = e^{-a^2t}\sin x+x^2 + 2a^2t
\end{align}
through which we can obtain the expression of $(f(t),u_0(x),\varphi(t))$ and the inversion input $g(t):=u(0, t)$.
We add random Gaussian noise on $g$ to produce noisy data $g^\d$ satisfying $\|g^\d-g\|_{L^2(0, T)}\le \d$ with various noise
level $\d>0$. We will use $g^\d$ to reconstruct $\sigma^\dag$. In order to capture the feature of the sought Robin coefficient,
we take $\Theta$ to be the form (\ref{TV}) with $\beta = 1$. We will use the initial guess $\xi_{0}=0$ to carry out the computation.
The parameters $\bar \mu_0$ and $\bar \mu_1$ in the definition of $\mu_n^\d$ are taken to be
$\bar \mu_0 = 1.8(1-1/\tau)/\beta$ and $\bar \mu_1 = 20000$. For implementing TPG-DBTS method with $\lambda_n^\d$ chosen by Algorithm
\ref{DBTSAlgorithm}, we take $j_{\max} = 2$, $\a=5$, $\gamma_0 = 0.1$ in (\ref{lambda2}), $\gamma_1 = 0.4$; we also choose the function
$q: {\mathbb N}\to {\mathbb N}$ by $q(m) = m^{-1.1}$. For implementing the Nesterov acceleration of Landweber iteration, we take
$\lambda_n^\d = n/(n+\a)$ with $\a = 5$. During the computation, the initial-boundary value problems for parabolic equation are
transformed into integral equations by the potential theory (\cite{K1989}) and then solved by a boundary element method by dividing
$[0, T]$ into $N = 64$ subintervals of equal length. The total variation denoising problem (\ref{TVdenoising}) involved in each iteration
step is solved approximately by the PDHG method after 200 iterations.

In Table \ref{table3} we report the computational results by TPG-DBTS, Landweber, and Nesterov acceleration of Landweber, using noisy data
for various noise level $\d>0$, which clearly demonstrates the acceleration effect of TPG-DBTS and Nesterov acceleration and shows that
these two methods have superior performance over Landweber iteration. In Figure \ref{fig3} we also plot the computational results by TPG-DBTS
using noisy data with noise level $\d = 0.001$. We note that the combination parameter $\lambda_n^\d$ produced by TPG-DBTS may be different from
$n/(n+\a)$ for some $n$, but eventually $\lambda_n^\d$ becomes the same as the combination parameter $n/(n+\a)$ in Nesterov acceleration.

\begin{figure}[htp]
  \begin{center}
  \includegraphics[width = 1 \textwidth]{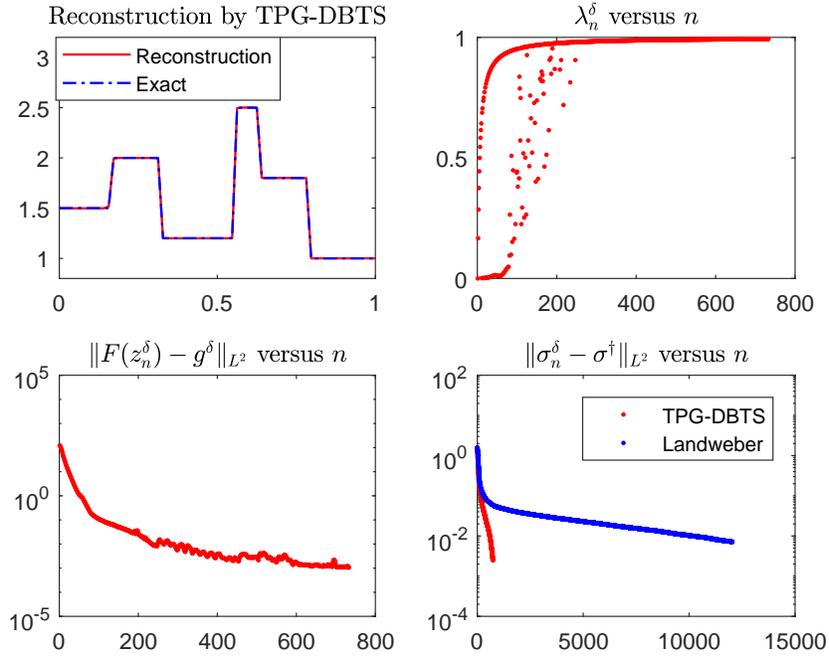}
  \end{center}
  \caption{Robin coefficient reconstruction with $\delta = 0.001$.} \label{fig3}
\end{figure}

\section*{\bf Acknowledgement}

The work of M Zhong is partially supported by the National Natural Science Foundation of China (No. 11501102) and Natural Science Foundation of Jiangsu Province (No. BK20150594).
The work of W Wang is partially supported by the National Natural Science Foundation of China (No. 11871180) and Natural Science Foundation of Zhejiang Province (No. LY19A010009).
The work of Q Jin is partially supported by the Future Fellowship of the Australian Research Council (FT170100231).

\end{document}